\documentclass[letterpaper, 11pt]{article}
\usepackage{amsmath}
\usepackage{amssymb}
\usepackage{latexsym}
\usepackage{color}

\definecolor{Red}{rgb}{1,0,0}
\definecolor{Blue}{rgb}{0,0,1}
\definecolor{Olive}{rgb}{0.41,0.55,0.13}
\definecolor{Green}{rgb}{0,1,0}
\definecolor{MGreen}{rgb}{0,0.8,0}
\definecolor{DGreen}{rgb}{0,0.55,0}
\definecolor{Yellow}{rgb}{1,1,0}
\definecolor{Cyan}{rgb}{0,1,1}
\definecolor{Magenta}{rgb}{1,0,1}
\definecolor{Orange}{rgb}{1,.5,0}
\definecolor{Violet}{rgb}{.5,0,.5}
\definecolor{Purple}{rgb}{.75,0,.25}
\definecolor{Brown}{rgb}{.75,.5,.25}
\definecolor{Grey}{rgb}{.5,.5,.5}
\definecolor{Black}{rgb}{0,0,0}

\def\path{{\tt path}}

\newcommand{\fcal}{\mathcal{F}}

\newcommand{\ical}{\mathcal{I}}

\newcommand{\kcal}{\mathcal{K}}

\newcommand{\qcal}{\mathcal{Q}}

\newcommand{\real}{\mathbb{R}}

\newcommand{\eps}{\varepsilon}

\newcommand{\bdm}{\begin{displaymath}}
\newcommand{\edm}{\end{displaymath}}
\newcommand{\bea}{\begin{eqnarray*}}
\newcommand{\eea}{\end{eqnarray*}}
\newcommand{\bean}{\begin{eqnarray}}
\newcommand{\eean}{\end{eqnarray}}

\newcommand{\expec}{\mathbb{E}}

\newcommand{\diff}{\mathrm{d}}

\newcommand{\onebf}{\mathbf{1}}
\newcommand{\tilv}{\tilde v}
\newcommand{\undfcal}{\underline{\fcal}}
\newcommand{\ovfcal}{\overline{\fcal}}

\newcommand{\undf}{\underline{f}}
\newcommand{\ovf}{\overline{f}}

\newcommand{\undff}{\underline{F}}
\newcommand{\ovff}{\overline{F}}
\newcommand{\undnu}{\underline{\nu}}
\newcommand{\ovnu}{\overline{\nu}}
\newcommand{\undmm}{\underline{M}}
\newcommand{\ovmm}{\overline{M}}
\newcommand{\undrho}{\underline{\rho}}
\newcommand{\ovrho}{\overline{\rho}}
\newcommand{\undnn}{\underline{N}}
\newcommand{\ovnn}{\overline{N}}
\newcommand{\undeta}{\underline{\eta}}
\newcommand{\oveta}{\overline{\eta}}
\newcommand{\undtt}{\underline{T}}
\newcommand{\ovtt}{\overline{T}}
\newcommand{\undtau}{\underline{\tau}}
\newcommand{\ovtau}{\overline{\tau}}
\newcommand{\unddd}{\underline{D}}
\newcommand{\ovdd}{\overline{D}}
\newcommand{\undsigma}{\underline{\sigma}}
\newcommand{\ovsigma}{\overline{\sigma}}
\newcommand{\undlambda}{\underline{\lambda}}
\newcommand{\ovlambda}{\overline{\lambda}}
\newcommand{\undss}{\underline{S}}
\newcommand{\ovss}{\overline{S}}
\newcommand{\undrr}{\underline{R}}
\newcommand{\ovrr}{\overline{R}}
\newcommand{\tildeqcal}{\widetilde{\qcal}}
\newcommand{\tildeq}{\tilde{q}}
\newcommand{\rmif}{\mathrm{if\ }}
\newcommand{\rmow}{\mathrm{o.w.}}
\newcommand{\ind}{\mathbf{1}}

\newcommand{\tildelambda}{\tilde{\lambda}}
\newcommand{\tildenu}{\tilde{\nu}}

\newtheorem{theorem}{Theorem}
\newtheorem{proposition}{Proposition}

\newtheorem{definition}{Definition}
\newtheorem{example}{Example}
\newtheorem{lemma}{Lemma}

\newenvironment{proof}{\noindent{\textbf{Proof:}}}{$\blacksquare$\vskip\belowdisplayskip}
\newenvironment{prevproof}[2]{\noindent {\bf {Proof of
{#1}~\ref{#2}:}}}{$\blacksquare$}

\title{First to Market is not Everything:\\ 
an Analysis of Preferential Attachment with Fitness}

\author{Christian Borgs\footnote{Microsoft Research. {\tt email:
borgs@microsoft.com}} \and Jennifer~T.~Chayes\footnote{Microsoft
Research. {\tt email: jchayes@microsoft.com}} \and
Constantinos Daskalakis\footnote{Computer Science, U.C. Berkeley.
Research done while the author was visiting Microsoft
Research.{\tt email: costis@cs.berkeley.edu}}
\and S\'ebastien Roch\footnote{Department of Statistics,
U.C.Berkeley. Research done while the author was visiting
Microsoft Research. {\tt email:
sroch@stat.berkeley.edu}}}

\addtocounter{page}{-1}

\begin{document}

\maketitle

\thispagestyle{empty}

\begin{abstract}

{The design of algorithms on complex
networks, such as routing, ranking or recommendation algorithms,
requires a detailed understanding of the growth
characteristics of the networks of interest, such as the Internet,
the web graph, social networks or online communities.} To this end,
preferential attachment, in which the popularity (or relevance) of
a node is determined by its degree, is a well-known and appealing
random graph model, whose predictions are in accordance with
experiments on the web graph and several social networks. However,
its central assumption, that the popularity of the nodes depends
only on their degree, is not a realistic one, since every node has
potentially some intrinsic quality which can differentiate its
attractiveness from other nodes with similar degrees.

{In this paper, we provide a rigorous analysis of {\em preferential
attachment with fitness}, suggested by Bianconi and Barab\'asi and
studied by Motwani and Xu, in which the degree of a
vertex is scaled by its quality to determine its attractiveness.}
Including quality considerations in the classical preferential
attachment model provides a much more realistic description of
many complex networks, such as the web graph, and allows to
observe a much richer behavior in the growth dynamics of these
networks. Specifically, depending on the shape of the distribution
from which the qualities of the vertices are drawn, we observe
three distinct phases, namely a {\em first-mover-advantage} phase, a
{\em fit-get-richer} phase and an {\em innovation-pays-off}
phase. We precisely characterize the properties of the quality
distribution that result in each of these phases and we compute
the exact growth dynamics for each phase. The dynamics provide
rich information about the quality of the vertices, which can be
very useful in many practical contexts, including ranking
algorithms for the web, recommendation algorithms, as well as the
study of social networks. Furthermore, the mathematical techniques
we introduce to establish these dynamics could be applicable to a
wide variety of problems.


\end{abstract}

\bigskip

\noindent\textbf{Keywords:} random graphs, preferential attachment, P\'olya urn
processes, Bose-Einstein condensation

\newpage

\section{Introduction}

{
In recent years, there has been a convergence of ideas coming from
computer science, social sciences and economic sciences as researchers in these fields
attempt to model and analyze the characteristics and dynamics 
of large complex networks, such as the web graph, social networks
and recommendation networks. From the computational perspective,
it has been recognized that the successful design of algorithms performed on
such networks, including routing, ranking and recommendation
algorithms, must take into account the social dynamics as well as
the technical properties and economic incentives that govern
network growth
\cite{DBLP:conf/stoc/Papadimitriou01,DBLP:conf/stoc/Raghavan06,DBLP:conf/focs/Kleinberg06}.}

\vspace{-0.3cm}\paragraph{Random Graph Models.} An appealing way to model the growth dynamics of these networks is
via random graph models. The well-studied Erd\"os-R\'enyi model is
not an appropriate description of these networks, because it is a
static rather than dynamic model, and more importantly, because
sparse graphs drawn from the Erd\"os-R\'enyi model have Poisson
degree distributions rather than the scale-free (power-law)
distributions observed in a variety of social phenomena
\cite{Zipf:49}, and verified by experiments on the World Wide Web
\cite{barabasi-1999-286,
DBLP:conf/sigcomm/FaloutsosFF99,DBLP:conf/cocoon/KleinbergKRRT99}---the 
latter seen as a massive graph with web pages being its
vertices and directed edges between vertices corresponding to
hyperlinks from one page to another.

Several models have been suggested which result in scale-free
distributions, probably the first being due to Yule~\cite{Yule:25}
and Simon~\cite{Simon1955}. In the context of scientific citations
power law distributions were observed by Lotka~\cite{Lotka1926},
and Gilbert~\cite{Gilbert1997} specifies a probabilistic model
supporting Lotka's law. Kleinberg et
al.~\cite{DBLP:conf/cocoon/KleinbergKRRT99} and Kumar et
al.~\cite{DBLP:conf/focs/KumarRRSTU00} suggest and study the {\em
copy model} which captures the power law distribution and
other connectivity properties of the World Wide Web, while other
models include works from Broder et al.~\cite{Broder&etal2000a},
Cooper and Frieze~\cite{Cooper&Frieze2001}, Drinea et al.~\cite{drinea@variations}, 
Krapivsky and Redner~\cite{Krapivsky&Redner2001}.

\vspace{-0.3cm}\paragraph{Preferential Attachment Models.}
One of the most natural and attractive models for network growth
is the {\em preferential attachment model}, suggested by
Barab\'asi and Albert~\cite{barabasi-1999-286} to model the web
graph, and originally proposed as the {\em cumulative advantage}
model by Derek de Solla Price in 1965~\cite{Price1965}. 
See e.g.~\cite{BollobasRST01, BollobasR04} for a rigorous treatment. Roughly
speaking, as time evolves, new vertices join the network by adding
several links to the vertices already present in the network in a
probabilistic fashion. The probability of linking to an existing
vertex is an increasing function, usually polynomial, in its
degree, which captures the intuitive fact that higher degree of a
vertex reflects higher relevance or popularity.

This model by itself has been rather successful in predicting the
graph structure of the web \cite{barabasi-1999-286}, at least as
an undirected graph. Nevertheless, there is an unsatisfactory
assumption underlying the model. 
{The popularity of a vertex
depends only on its degree. As a result, the prediction
of the model is the so-called  {\em first-mover-advantage}
phenomenon in which earlier vertices tend to have
significantly higher degrees than later ones, making it
hard for a vertex which enters late to compete with the
already established {\em hubs} of the network.} Moreover, the model
is completely symmetric with respect to vertices which enter at
similar times, since there is no modeling of how the intrinsic
quality of every vertex affects its growth in the network. 
{How is the quality
of vertices reflected in the
network structure and its dynamics?
How can one extract such information?}

{
To answer this type of questions we analyze a variant
of the preferential attachment model which explicitly models the
intrinsic quality of the vertices.} This model, introduced in the
context of the web by Bianconi and Barab\'asi \cite{BianconiBarabasi:01},
is usually called {\em preferential attachment with
fitness}. In this model, when a new vertex is created, it gets
assigned a quality parameter, henceforth called {\em fitness},
drawn from a given distribution, which scales its degree
to determine its attractiveness in the evolution of the network.
The resulting model provides a much more accurate description of
many real-world networks \cite{BianconiBarabasi:01}, but it is
also more difficult to analyze rigorously; see Bianconi and
Barab\'asi \cite{BianconiBarabasi:01} for heuristic arguments and Motwani
and Xu \cite{DBLP:conf/pods/MotwaniX06} for more precise---but
nevertheless heuristic in several aspects---arguments.

\vspace{-0.3cm}\paragraph{Our Results.} We provide the first---to our knowledge---rigorous analysis of
preferential attachment with fitness. We show that, depending on
the properties of the distribution from which the fitnesses are
drawn, henceforth called the {\em fitness distribution}, there is
a much richer behavior that an evolving network may exhibit than
what is predicted by the classical preferential attachment model.
We precisely characterize the possible evolutions of a complex
network and we specify the properties of the fitness distribution
resulting in each of them. More precisely, we show that, depending
on the fitness distribution, an evolving network can undergo one
of the following behaviors, or {\em phases}:
\begin{itemize}
{
\item the {\em first-mover-advantage} phase, which results
from flat fitness distributions and corresponds to the power-law
behavior predicted by the classical preferential attachment model;

\item the {\em fit-get-richer} phase, in which vertices of
higher fitness grow faster than those of smaller fitness; the behavior here is a
power-law within each fitness value, but the
tail exponent decreases as the fitness increases;

\item the {\em innovation-pays-off} phase, in which roughly
speaking the competition for links results in a
constant fraction of the links continuously shifting
to ever larger fitness values;  this fraction of links
that ``escapes to infinity'' is independent of the network
size and is determined by the fitness distribution; such
behavior is not observed in the fit-get-richer phase.}
\end{itemize}

Our analysis is applicable to both discrete and continuous fitness
distributions, as well as bounded or unbounded ones, and we
provide precise criteria for the fitness distribution that specify
which of the above phases will arise. In fact, we discover some
property of the fitness distribution which exhibits a sharp phase
transition separating the latter evolution scenarios. Our results
are in accordance with the predictions of Bianconi and Barab\'asi
\cite{BianconiBarabasi:01} derived by mapping the evolving network
to a Bose gas in the thermodynamic limit. In this terminology, the
innovation-pays-off phase corresponds to the phenomenon of
{\em Bose-Einstein condensation}, whereby a constant fraction of
the particles condensate on the lowest energy level, corresponding
in the network context to the supremum of the fitness values.

A by-product of our technique is a
precise characterization of the {\em vertex dynamics} under
preferential attachment with fitness. More specifically, if a
vertex $v$ has fitness $f$, then our analysis implies that its
degree $d_v(t)$ at time $t$ scales as
\begin{equation}\label{eq:vertexdynamics}
d_v(t)\sim t^{{cf}},
\end{equation}
where $c$ is a
global constant determined by the fitness distribution. Hence, the
logarithm of the degree of the vertices directly reflects their
quality. This could suggest new directions in the design of ranking or
recommendation algorithms.

{
\vspace{-0.3cm}\paragraph{Proof Techniques.} 
The standard approach to analyze preferential attachment models
is to derive recursions (or differential equations), 
typically, of the expected number
of nodes of a given degree. See e.g.~\cite{Mitzenmacher:04}.
This type of technique relies
crucially on the fact that the number of nodes at any time
in the graph is deterministic---a quantity that arises
as the denominator in the recursion. However, in our case,
the relevant quantity is the number of nodes \emph{weighted
by their fitness} which, unfortunately, is a
\emph{random} variable. This turns out to complicate
significantly the analysis.
}

{
To obtain our results, 
we rely instead on a very different approach, one based
on the theory of P\'olya urn models. 
In P\'olya's classical urn scheme, an urn contains balls of two colors.
At each time step, a ball is drawn randomly from the urn
and returned along with an extra ball of the same color. 
This is clearly reminiscent of a preferential attachment
scheme and the connection between the two
models has previously been exploited, e.g. in~\cite{BeBoChSa:05}.
Here we use a generalized version of P\'olya's scheme (see e.g.~\cite{Janson:04}): 
1) we consider an arbitrary, but finite number of colors;
2) each ball is picked proportionally to a weight, or ``activity parameter'', associated to its color;
and 3) at each time step, the ball picked is returned along with a random number of balls of each
color, where the distribution of this ``random update vector'' depends on the color
of the ball drawn. 

We analyze the limiting behavior of the preferential attachment scheme with
fitness by coupling the growth process with specially crafted generalized
P\'olya urn models where the colors represent
connectivity properties of the evolving network, e.g.~the
cumulative degree of all vertices of a given fitness.
When the fitness distribution is concentrated on a finite number of atoms,
the correspondence is somewhat straightforward, although our coupling
appears to be novel and it allows to derive nontrivial generalizations of classic results very easily. 
More importantly,
we consider in fact general fitness distributions, including continuous
distributions, which in principle require an infinite number of colors
in the P\'olya urn model. 
Little is known about the behavior of generalized P\'olya urns
beyond the finite case, and we resort to various novel truncation
techniques to map the dynamics of our network to a finite urn process.
We expect that our techniques should be useful in a much more general context to the
analysis of previously unapproachable complex network growth models, which now may
be analyzed using infinite P\'olya urn models with techniques analogous to those
developed here.}

{

\subsection{Definitions and Main Result}

\vspace{-0.0cm}\paragraph{The Model.} 
The generalized preferential attachment model of Bianconi and
Barab\'asi which we analyze here is a random graph model 
defined as follows. 

\vspace{-0.3cm}\begin{definition}[Preferential Attachment Scheme with Fitness]
Let $\fcal \subseteq \real_+$ be a set of fitnesses and $\qcal$ a
distribution over fitnesses such that $\int_{\fcal}\diff \qcal(f)=1$.
%
%
The \emph{preferential attachment process with fitness} 
begins with one vertex of fitness $f\in \fcal$ drawn
according to $\qcal$ and a self-loop on that vertex. Then, at
every time step $t$, a new vertex is added to the graph, which has
fitness picked independently according to $\qcal$ and is attached
to an old vertex $v$ with probability proportional to $f_v \cdot
d_{v,{t-1}}$, where $f_v$ is the fitness of vertex $v$ and
$d_{v,{t-1}}$ its degree at step $t-1$. We denote by
$G_n = (V_n, E_n)$ the graph at time $n$.
We sometimes refer to this
process as the {\em $(\fcal,\qcal)$-chain}.
\end{definition}

\vspace{-0.3cm}\noindent It turns out that the case of
unbounded fitnesses is rather uninteresting 
(see Appendix~\ref{sec:app:continuous_unbounded}) 
and hereon we
assume 
that $\sup\{f\,:\, f\in \fcal\} = h$ for some $h < +\infty$. 
Furthermore, we consider three main cases for $\fcal$:
either $\fcal$ is discrete---finite or countable---with
$\qcal$ strictly positive on $\fcal$,
or
$\fcal$ is the interval $[0,h]$ and $\qcal$ admits a strictly positive continuous
density on $(0,h)$. We say that $(\fcal, \qcal)$ is \emph{regular} in such cases.
Our results extend to more general fitness distributions but we restrict
ourselves to the regular case here. Also, the process above constructs only undirected trees.
However, our techniques can be easily extended to directed scale-free graphs as defined
in~\cite{BoBoChRi:03}. We omit the details.

\vspace{-0.3cm}\paragraph{Main Result.} 
Our basic result concerns the distribution of links across fitnesses
as $n\to +\infty$. Let $[a,b] \subseteq [0,h]$ with $a\leq b$ and denote
by $M_{n,[a,b]}$ the number of edge endpoints with fitness in
$[a,b]$ in $G_n$. Let $\lambda_0$ be the (unique) solution in $[h,+\infty)$
of
\begin{equation}\label{eq:intro_be}
\ical(\lambda_0) \equiv \int_\fcal \frac{f}{\lambda_0-f}\diff \qcal(f) = 1,
\end{equation} 
if it exists and let $\lambda_0 = h$ otherwise.
Our main result is the following.
\begin{theorem}[Basic Result]\label{thm:intro_basic}
Assume $(\fcal,\qcal)$ is regular. Then, for all $[a,b] \subseteq [0,h)$
with $a\leq b$, we have
\begin{equation*}
\frac{M_{n,[a,b]}}{n} \to \lambda_0 \int_{\fcal\cap[a,b]} \frac{1}{\lambda_0-f}\diff \qcal(f) 
\equiv \nu_{[a,b]},
\qquad
\frac{M_{n,[a,h]}}{n} \to 2 - \lambda_0 \int_{\fcal\cap[0,a]} \frac{1}{\lambda_0-f}\diff \qcal(f)
\equiv 2 - \nu_{[0,a]},
\end{equation*}
almost surely as $n \to +\infty$.
\end{theorem}
A surprising behavior arises when (\ref{eq:intro_be}) has no solution in $[h,+\infty)$,
or equivalently when $\ical(h) < 1$. Indeed, in such a case, it is easy to check that
$\nu_{[0,h-\eps]} \leq 1 + \ical(h) < 2$ for all $\eps>0$ 
even though we expect $\lim_{\eps \to 0} \nu_{[0,h-\eps]} = 2$ since
for all $n$, $n^{-1}M_{n,[0,h]} = 2$ (i.e.~each edge has two endpoints).
In other words, it appears that a constant fraction of edges is ``missing'' in the limit.
The missing fraction actually ``escapes to $h$'' which leads to what we call
the innovation-pays-off phase as described above. To get a better intuition for the
existence of a solution in (\ref{eq:intro_be}), consider the example $\qcal \sim \mathrm{Beta}(\alpha, \beta)$.
In Example~\ref{example:beta} of Appendix~\ref{sec:becont}, we show there is a solution
if and only if $\beta\leq \alpha + 1$. For a fixed $\alpha$, a large $\beta$ indicates
a ``fast decay'' to $0$ at $1$ while a small $\beta$ leads to a ``fatter tail'' around $1$.
A solution to (\ref{eq:intro_be}) exists in the latter case, e.g.~in the uniform case.
In other words, the innovation-pays-off regime requires a more ``rarefied'' high fitness
population.

\vspace{-0.3cm}\paragraph{Dynamics of the Innovation-Pays-Off Phase.}
In order to understand (informally) the dynamics of the innova\-tion-pays-off phase, 
fix a time $t^*$
and let $f^*$ be the largest fitness among vertices
present in the network at time $t^*$. Note that
\begin{itemize}

\item at time $t^*$, the cumulative fraction of the links shared
by vertices of fitness up to $f^*$ is $2$,
since every edge is accounted for twice;

\item now, consider the network in the limit $t=+\infty$; by
Theorem~\ref{thm:intro_basic} and the discussion above, 
the fraction of links shared among vertices
of fitness up to $f^*$ is at most $1+\ical(h)$;
therefore at least a fraction $1-\ical(h)$ of links is shared
among vertices of fitness larger than
$f^*$, vertices which, by definition, were not present at time $t^*$.
\end{itemize}
This is
the ``signature'' of the innovation-pays-off phase: a constant
fraction of the links changes hands toward higher and higher
fitness values. 

\vspace{-0.3cm}\paragraph{Power Laws and Vertex Dynamics.} In fact, we can prove more than
Theorem~\ref{thm:intro_basic}. As stated below in Theorems~\ref{thm:fitgetrich}
and~\ref{thm:boseeinstein} and their counterparts in the continuous case, 
we exhibit power laws for the degree distributions on the nodes of a given
fitness and we get a tail exponent of $\lambda_0 f^{-1}$ where $f$ is the
given fitness. 
See Section~\ref{sec:countable}. Also, as discussed above, we can prove
vertex dynamics of the form (\ref{eq:vertexdynamics}).
Such result is proved by considering a continuous-time
embedding of the process as in~\cite{Janson:04}. Details are omitted.
The constant $c$ in (\ref{eq:vertexdynamics}) is in fact $\lambda_0^{-1}$.

\vspace{-0.3cm}\paragraph{Proof Sketch.} As we mentioned before, the 
basic idea of the proof of Theorem~\ref{thm:intro_basic}
(as well as of the power law results in Theorems~\ref{thm:fitgetrich}
and~\ref{thm:boseeinstein} below) is to couple
the preferential attachment process with P\'olya urn models. The first
step is the analysis of the case $\fcal$ finite. There we proceed by truncating
large degrees and associating a color of a specially designed P\'olya
process to each pair (degree, fitness). The limit theory of P\'olya processes
then reduces the problem to an eigenvector computation of an appropriately
defined matrix (see Section~\ref{sec:polya}). This computation appears to be
tricky but turns out to be manageable, as described in Appendix~\ref{sec:app:discrete_finite}. 

The countable and continuous cases are significantly more challenging since P\'olya urns
with infinite---whether countable or uncountable---colors are poorly understood. Instead,
we use further truncation and approximation techniques to couple
the infinite cases with finite cases. In Section~\ref{sec:countable}, we
illustrate this idea on the somewhat easier special case of $\fcal = \{f_j\}_{j=1}^{+\infty}$
increasing. There we need two finite P\'olya models---a lower bound and an upper bound---which 
are obtained by truncating $\fcal$ and mapping the remaining fitness values
to either $0$ or $h$. The general discrete case as well as the continuous case
require a much more sophisticated approach which is detailed in Appendix~\ref{sec:app:continuous}.

\vspace{-0.3cm}\paragraph{Organization of the Paper.}
We start with a brief overview of generalized P\'olya urn models
in Section~\ref{sec:polya} followed by our treatment of
preferential attachment for
finite fitness distributions in Section~\ref{sec:finitetype}. The main steps of the general proof
are illustrated in Section~\ref{sec:countable} in the special case where
$\fcal = \{f_j\}_{j\geq 1}$ is countable and increasing. Most proofs
are relegated to the appendix. Most notably, for lack of space the particularly
interesting analysis of the 
continuous case is {\em completely} relegated to Appendix~\ref{sec:app:continuous}.

\vspace{-0.3cm}\paragraph{Notation.}
We denote by $e_i$ the unit vector along the $i$-th axis
(usually the dimension is clear). The notation $\onebf_{S}$ denotes
the indicator of the event $S$.
}

\section{Generalized P\'olya Urns}\label{sec:polya}

{
Our results are obtained through an appropriate mapping of the preferential fitness
process to a finite generalized P\'olya urn scheme. 
We introduce here the basic limit theory of generalized P\'olya urn models keeping our
notation consistent with the presentation of
Janson~\cite{Janson:04}, with the exception of our matrix $A$
which is the transpose of Janson's, in accordance with common
practice in the P\'olya urn literature.}

\vspace{-0.5cm}\paragraph{Definition of the P\'olya Urn Process.} We have $q <
+\infty$ bins (corresponding to the colors in the original
P\'olya model described in the Introduction). 
Each bin $i\le q$ is assigned a fixed activity
$a_i$, $0 \leq a_i < +\infty$. For $n \geq 0$, let
\begin{equation*}
X_n = (X_{n,1},\ldots,X_{n,q}),
\end{equation*}
where $X_{n,i}$ is the number of balls in bin $i$
at time $n$. The initial load is given by $X_0$, which may be random or deterministic.
Each bin, say $i$, also has a random vector $\xi_i = (\xi_{i,1},\ldots,\xi_{i,q})$ with
integer coordinates.
The process is defined as follows. At time $n$, we pick one bin. Bin $i$ is chosen
with probability proportional to $a_i X_{n-1,i}$. If bin $i$ is picked, we draw
an independent copy $\xi^{(n)}_{i}$ of $\xi_{i}$ and update $\{X_n\}_{n\geq 0}$ according to
\begin{equation*}
X_n = X_{n-1} + \xi^{(n)}_{i}.
\end{equation*}

\vspace{-0.5cm}\paragraph{Basic P\'olya Urn Result.} The limiting behavior of
the P\'olya Urn process described above can be characterized in
terms of the $q\times q$ matrix $A$ with entries
\begin{equation*}
A_{i,j} = a_i \expec[\xi_{i,j}],
\end{equation*}
assuming conditions (A1)-(A6) in~\cite{Janson:04} are satisfied.
In fact, we will only need to use the more general assumption
described in Remark 4.2 of~\cite{Janson:04}. Roughly speaking, we
require that:
\begin{itemize}
\item The urn process is well-defined (see the definition of {\em
tenable} in Remark 4.2 of~\cite{Janson:04}). Essentially, we
require that the number of balls remains nonnegative at all times
with probability 1.

\item The matrix $A$ satisfies a slight generalization of
irreducibility and the initial load is positive on a ``dominating
type.'' This generalization allows for dummy bins that ``count
certain events.'' (See Section 3 ``Limits for urns''
of~\cite{Janson:04}.)

\item The vectors $\xi_i$ have finite second moments. In our
application, the $\xi_i$'s will actually be bounded.
\end{itemize}
We refer the reader to~\cite{Janson:04} for more details. Under
these conditions, it is not hard to see that $A$ has a unique
largest positive eigenvalue $\lambda_1$ with corresponding
positive left eigenvector $v_1$ and right eigenvector $u_1$ (apply
the Perron-Frobenius theorem to $A + \alpha I$ for an appropriate
$\alpha$). We choose $u_1, v_1$ to satisfy $a\cdot v_1 = 1$ and
$u_1 \cdot v_1 = 1$ where $a$ is the vector of activities. The
following theorem characterizes the vector $X_n$.
\begin{theorem}[Limit of Finite Urns~\cite{AthreyaNey:72}; Theorem~3.21 in~\cite{Janson:04}]\label{thm:basic}
Assume conditions (A1)-(A6) of \cite{Janson:04} are satisfied.
Conditioned on essential non-extinction (see~\cite{Janson:04}) we
have
\begin{equation*}
\frac{X_n}{n} \to \lambda_1 v_1,
\end{equation*}
almost surely as $n \to +\infty$.
\end{theorem}
In our applications of Theorem \ref{thm:basic}, it will be easy to
establish that ``essential extinction'' is not possible.

\section{Preferential Attachment: Finite Distributions}\label{sec:finitetype}

{
In this section, we treat the case $\fcal = \{f_j\}_{j \in J}$
where $J$ is finite---which we sometimes refer to as the {\em finite-type case}. 
This will form the basic step in the analysis
of the countable
and continuous cases. Without loss
of generality, we take $\{f_j\}_{j \in J}$ increasing.
We analyze separately the distribution of degrees within each fitness value
(Section \ref{sec:first-mover-advantage}) and the distribution of links 
across
fitness values (Section \ref{sec:among_fitnesses}). We
then combine the two results in Section
\ref{sec:full}. Note that, as we describe below, only the first-mover-advantage
and fit-get-richer behaviors arise in the finite-type case. 
}


\subsection{Flat Fitness Distributions: First-Mover-Advantage}
\label{sec:first-mover-advantage}

Suppose first that $J=1$. This is the standard preferential attachment
model, which is well understood (see e.g.~\cite{Mitzenmacher:04}
and references therein). We rederive the degree distribution by
first mapping to a P\'olya urn process and then applying
Theorem~\ref{thm:basic}. The mapping is illustrative of our
technique. Let $L_{n,k}$ be the number of vertices of degree $k$
at time $n$; set $\mu_1 = \frac{2}{3}$ and, for $k \geq 2$,
\begin{equation*}
\mu_k = \frac{2}{3}\prod_{l=2}^k \frac{l-1}{l+2} =
\frac{4}{k(k+1)(k+2)} \sim k^{-3}.
\end{equation*}
In particular, $\{\mu_k\}_{k\geq 1}$ is a {\em power law with tail
exponent $2$}.
\begin{proposition}[1-Fitness Case; see e.g.~\cite{Mitzenmacher:04}]\label{prop:1fit}
For all $k \geq 1$,
\begin{equation*}
\frac{L_{n,k}}{n} \to \mu_k
\end{equation*}
almost surely as $n \to +\infty$.
\end{proposition}
\begin{proof}
Fix $k \geq 1$ and consider the following urn process with $k+1$
urns of equal activities $a_i = 1$, for all $1\leq i\leq k+1$. We
will design the process in such a way that the number of balls in
urn $i$ at time $n$ represents the number of edges in the graph
which are adjacent to vertices of degree $i$---counting twice
edges with both endpoints at vertices of degree $i$.
Except for the $(k+1)$-st urn, where the number of balls will
represent the number of edges adjacent to vertices of degree $\ge
k+1$.

Let $X_0 = (0,2,0,\ldots,0)$ reflecting the fact that initially
there is a single vertex with a self loop (degree $2$). For $2
\leq i \leq k$, let the update vector $\xi_i$ be deterministic
with
\begin{equation*}
\xi_{i,j} = \left\{\begin{array}{ll}
1, & j=1\\
-i, & j=i\\
i+1, & j=i+1\\
0, &\mathrm{o.w.}
\end{array}
\right.
\end{equation*}
reflecting the fact that, if the new vertex being added to the
graph links to an old vertex of degree $i$, then the degree of
that vertex becomes $i+1$, therefore the edges adjacent to that
vertex must be accounted for in the urn $i+1$ instead of the urn
$i$. Finally, for urns $i=1$ and $i=k+1$, the following update
vectors respect the boundary conditions
\begin{align*}
\xi_{1,j} = \left\{\begin{array}{ll}
2, & j=2\\
0, &\mathrm{o.w.}
\end{array}
\right.\text{~~~~~~~~~~and~~~~~~~~~~}\xi_{k+1,j} =
\left\{\begin{array}{ll}
1, & j=1\\
1, & j=k+1\\
0, &\mathrm{o.w.}
\end{array}
\right.
\end{align*}

It is not hard to see that the urn process described above can be
coupled with the preferential attachment process so that with
probability $1$ the following relations are satisfied, for all
$n\ge 0$,
\begin{equation*}
\left\{\begin{array}{ll}
X_{n,\ell} = \ell L_{n,\ell}, &\text{ for }1\leq \ell \leq k\\
X_{n,k+1} = \sum_{\ell\geq k+1} \ell L_{n,\ell} &~
\end{array}
\right.
\end{equation*}
The proof is concluded by computing matrix $A$, its largest
eigenvalue $\lambda_1$ and the corresponding left eigenvector
$v_1$ (see Appendix \ref{sec:app:discrete_finite}). One can check
that Conditions (A1)-(A6) of \cite{Janson:04} are satisfied.
\end{proof}

\subsection{Competition for Links across Fitness Values} \label{sec:among_fitnesses}

We now consider the case $J = |\fcal| > 1$ finite. We aim to
compute the limiting behavior of the random variables $M_{n,j}$,
$1\leq j \leq J$, corresponding to the number of edges with an
endpoint of fitness $f_j$ at time $n$---counting twice edges with
two endpoints of fitness $f_j$, i.e. the total degree of vertices
of fitness $f_j$. Let $\lambda_0 > 0$ be the largest solution to the
equation
\vspace{-0.2cm}\begin{equation}\label{eq:befinite}
\sum_{j=1}^J \frac{f_j q_j}{\lambda_0 - f_j} = 1,
\end{equation}

\vspace{-0.3cm}\noindent where, by monotonicity, $\lambda_0 \in (\max_j{\{f_j\}},+\infty)$.
Also, for $1 \leq j \leq J$, set
\vspace{-0.2cm}\begin{equation}\label{eq:definition of nu}
\nu_{j} = \lambda_0 \frac{q_j}{\lambda_0 - f_j},
\end{equation}

\vspace{-0.3cm}\noindent and verify that
\vspace{-0.2cm}\begin{equation*}
\sum_{j=1}^J \nu_j =
\sum_{j=1}^J (\lambda_0-f_j) \frac{q_j}{\lambda_0 - f_j} +
\sum_{j=1}^J f_j \frac{q_j}{\lambda_0 - f_j} = 2.
\end{equation*}

\vspace{-0.2cm}\noindent We characterize the distribution of links across fitness values in
terms of the $\nu_j$'s.
\begin{proposition}[Fitness Alone]\label{prop:alone}
For all $1 \leq j\leq J$,
\begin{equation*}
\frac{M_{n,j}}{n} \to \nu_j,
\end{equation*}
almost surely as $n \to +\infty$.
\end{proposition}
\begin{proof}
We define the following urn process with $J$ urns in which urn $i
\le J$ has activity $a_i = f_i$. The urn process will be designed
so that the number of balls in urn $i$ corresponds to the number
of edges with an endpoint of fitness $f_i$. For $1 \leq i\leq J$,
the update vector $\xi_i$ is given by $\xi_i = e_i + \Delta_i$,
where $\Delta_i = e_j$ with probability $q_j$, for all $1 \leq
j\leq J$. In the context of the preferential attachment process,
this reflects the fact that, if the new vertex links to a bin of
fitness $f_i$, then the number of edges with an endpoint of
fitness $f_i$ increases by one, hence the term $e_i$; moreover,
the new vertex picks a random fitness according to $\qcal$, hence
the term $\Delta_i$. It is easy to couple the defined urn process
with the preferential attachment one so that, with probability
$1$, $X_{n,j} = M_{n,j}$, for all $1\leq j\leq J$ and all $n \geq
0$, provided $X_0 = 2e_i$ with probability $q_i$. The proof is
concluded by computing matrix $A$, its largest eigenvalue
$\lambda_1$ and the corresponding left eigenvector $v_1$ (see
Appendix \ref{sec:app:discrete_finite}).
\end{proof}

\subsection{Finite Distributions: Fit-Get-Richer}\label{sec:full}

{
In this section, we derive the degree distribution of preferential attachment with
fitness under finite fitness distributions.}
For all $1\leq j\leq J$ and $k \geq 1$, denote by $N_{n,(j,k)}$
the number of vertices of fitness $f_j$ and degree $k$ at time
$n$. Define $\lambda_0$ and $\{\nu_j\}_{j=1}^J$ as in Section
\ref{sec:among_fitnesses}. 
Moreover, for all $1\leq j\leq J$ and $k \geq 1$, set
$\eta_{(j,k)}$ as follows
\begin{equation} \label{eq:basic_fit_get_rich}
\eta_{(j,k)} = \nu_j \cdot
\frac{1}{k}\prod_{\ell=2}^{k}\frac{\ell}{\ell + \lambda_0
f_{j}^{-1}}.
\end{equation}
In particular,
\begin{equation*}
\frac{\eta_{(j,k+1)}}{\eta_{(j,k)}} = \frac{k}{k+1}\frac{k+1}{k+1 + \lambda_0 f_{j}^{-1}}
= 1 - \frac{1+ \lambda_0 f_{i}^{-1}}{k}(1+o(1)),
\end{equation*}
as $k$ gets large. Thus, for fixed $j$, $\{\eta_{(j,k)}\}_{k\geq
1}$ has tail exponent $\lambda_0 f_j^{-1}$.
\begin{proposition}[Finite Fitness Distributions: Fit-Get-Richer]\label{prop:finitetype}
For all $1\leq j\leq J$ and $k \geq 1$, we have
\begin{equation*}
\frac{N_{n,(j,k)}}{n} \to \eta_{(j,k)},
\end{equation*}
almost surely as $n \to +\infty$.
\end{proposition}
{Observe that the tail exponent is a decreasing function of the
fitness. Hence, the tail of the distribution gets fatter as the
fitness increases. This is the ``signature'' of the
fit-get-richer phase. The proof of Proposition~\ref{prop:finitetype} is
postponed to the appendix. It follows from a
combination of the couplings in Propositions \ref{prop:1fit} and
\ref{prop:alone}, by defining a P\'olya urn process with a bin for every
pair of fitness and degree. Once again, the degree is truncated at a 
maximum value and an extra bin accounts for all degrees above.}

\vspace{-0.3cm}\section{Preferential Attachment: Countable Distributions} \label{sec:countable}

{
If $J=+\infty$, which we sometimes call the infinite-type case,
the coupling described in the previous section cannot be used
directly, since it would then require an infinite number of urns
(for the fitnesses alone) and Theorem~\ref{thm:basic} is not
known to hold generally in the infinite case.
Nevertheless, we obtain similar results 
by coupling our process this time with {\em two} finite-type preferential
attachment processes which
provide lower and upper bounds on the degree distribution of our
process. The coupling is presented in Section
\ref{sec:coupling_discrete}. Using this coupling and
Proposition~\ref{prop:finitetype}, we exhibit the
following evolution scenarios for the preferential attachment
process with countable fitness distribution:
\begin{itemize}
\item the {\em fit-get-richer scenario}, taking place
when $\sum_{j=1}^{+\infty} \frac{f_j q_j}{h - f_j} \ge 1,$

\item the {\em innovation-pays-off scenario}, taking place
when $\sum_{j=1}^{+\infty} \frac{f_j q_j}{h - f_j} < 1,$
\end{itemize}
where $h=\sup_{j\geq 1}{\{f_j\}}$. 

For convenience, we treat only the case $\{f_j\}_{j\geq 1}$ 
increasing. The general case---which is omitted
from this extended abstract---follows from an analysis similar
to that for continuous fitness distributions in Appendix~\ref{sec:app:continuous}.
}

\subsection{Coupling} \label{sec:coupling_discrete}

Denoting by $h$ the supremum of $\{f_j\}_{j\geq 1}$, let us assume
that $h < +\infty$; the case $h=+\infty$ is treated in
Section~\ref{sec:unbounded} of the appendix. Setting $I$ to be a
positive integer, the {\em upper I-truncation} of $\fcal$, denoted
$\ovfcal = \{\ovf_j\}_{j\geq 1}$, and the {\em lower I-truncation}
of $\fcal$, denoted $\undfcal = \{\undf_j\}_{j\geq 1}$, are
defined by
\begin{equation*}
\ovf_j =
\left\{\begin{array}{ll}
f_j, & j\leq I\\
0, & \mathrm{o.w.}
\end{array}
\right. ~~~~~~~~~~ \undf_j = \left\{\begin{array}{ll}
f_j, & j\leq I\\
h, & \mathrm{o.w.}
\end{array}
\right.
\end{equation*}
We shall couple the $(\fcal,\qcal)$ chain with the chains
$(\ovfcal,\qcal)$, $(\undfcal,\qcal)$ defined by the upper and
lower truncations to provide upper and lower bounds respectively
on the degrees of chain $(\fcal,\qcal)$\footnote{Strictly speaking,
we think of $\qcal$ here as a distribution on the
{\em indices} of the fitness sequences $\fcal$, $\undfcal$, and
$\ovfcal$ rather than on the fitnesses themselves.}. Roughly speaking, the
chains can be coupled so that, at every step, the probability of
choosing an old vertex of fitness value $f_1$ up to $f_J$ is
larger in the $(\ovfcal,\qcal)$ than in the $(\fcal,\qcal)$ chain
and larger in the $(\fcal,\qcal)$ than in the $(\undfcal,\qcal)$
chain. This property certainly holds in the beginning of the
processes and then reproduces itself since it makes the cumulative
degree of fitness levels $f_1$ up to $f_J$ grow faster in the
$(\ovfcal,\qcal)$ than in the $(\fcal,\qcal)$ chain and faster in
the $(\fcal,\qcal)$ than in the $(\undfcal,\qcal)$ chain. It is
important to note however that the degree by itself is not
sufficient to guarantee the domination of probabilities for the
next step of the process; rather we couple the edges which get
added at each step in such a way that the fitness values of the
endpoints in chain $(\undfcal,\qcal)$ dominate the fitness values
in $(\fcal,\qcal)$ and those dominate the fitness values in chain
$(\ovfcal,\qcal)$. 

\vspace{-0.5cm}\paragraph{Fitness Alone.} We first bound $M_{n,j}$, defined as in Section \ref{sec:among_fitnesses} to be the number of edges with an endpoint of
fitness $f_j$ (counting twice edges with two endpoints of fitness
$f_j$). Fixing $1 < I < +\infty$, let $\ovmm_{n,j}$ and
$\undmm_{n,j}$ be the corresponding variables of the
$(\ovfcal,\qcal)$, $(\undfcal,\qcal)$ chains. It is clear that the
latter are equivalent to finite type urn processes, so that
Proposition~\ref{prop:alone} applies. Let $\ovnu_j$ and
$\undnu_j$ be the (almost sure) limits of $n^{-1}\ovmm_{n,j}$ and
$n^{-1}\undmm_{n,j}$. Then we have the following.
\begin{lemma}[Coupling: Fitness Alone]\label{lem:couplingalone}
For all $1 \leq j \leq I$, it holds almost surely that
\begin{equation*}
\limsup_{n \to +\infty} \frac{M_{n,j}}{n} \leq
\ovnu_j,\text{~~~~~~and~~~~~~} \liminf_{n \to +\infty}
\frac{M_{n,j}}{n} \geq \undnu_j.
\end{equation*}
\end{lemma}
\begin{proof}
Consider the $(\fcal, \qcal)$-chain. At step $n \geq 1$, a vertex
is picked with probability proportional to its degree scaled by
its fitness. Let $F_n$ be the fitness of the chosen vertex and
denote by $\rho_{n-1,i}$ the probability that $F_n = f_i$ given
the state of the chain after step $n-1$. After a vertex is picked,
a new vertex is added with fitness chosen according to $\qcal$.
Let $F_n'$ be the fitness of this new vertex. Denote by $\ovff_n,
\ovff_n', \ovrho_n, \undff_n, \undff_n', \undrho_n$ the
corresponding variables for the chains $(\ovfcal,\qcal)$ and
$(\undfcal,\qcal)$ respectively. We define a coupling of the three
chains so as to preserve the following conditions:
\begin{enumerate}
\item For all $n \geq 1$, $\ovff_n \leq
F_n \leq \undff_n$ and $\ovff_n' \leq F_n' \leq \undff_n'.$

\item For all $n\geq 1$ and all $1\leq i\leq
I$, $\undmm_{n,i} \leq M_{n,i} \leq \ovmm_{n,i}.$

\item For all $n\geq 1$ and all $1\leq i\leq
I$, $\undrho_{n,i} \leq \rho_{n,i} \leq
\ovrho_{n,i}.$
\end{enumerate}
Note that 3. follows immediately from 1.~and 2. We now justify why
the conditions are satisfied for all $n\ge 0$. The initial
configuration ($n=0$) is constructed by picking an $i$ according
to $\qcal$ and choosing the corresponding fitness in all three
chains. Therefore the conditions are satisfied at time $0$ by the
definition of $\ovfcal$ and $\undfcal$. Assuming that Conditions
1., 2., and 3. are satisfied at time $n-1$ we will show that they
are true at time $n$. Indeed, since the fitness of the new vertex
is picked according to $\qcal$ in all $3$ chains it follows from
the definition of $\ovfcal$ and $\undfcal$ that $\ovff_n' \leq
F_n' \leq \undff_n'$. Now let us consider the step of picking the
old vertex. By 3., it follows that the choices made in the three
chains can be coupled so as to satisfy Conditions 1.~and 2.
Indeed, proceed as follows:
\begin{itemize}
\item with probability $\sum_{i=1}^I \undrho_{n-1,i}$, pick the
same fitness in all three chains according to
$\{(\undrho_{n-1,i})\}_{i=1}^I$;

\item with probability $\sum_{i=1}^I
(\rho_{n-1,i}-\undrho_{n-1,i})$, pick the same fitness in chains
$(\fcal,\qcal)$ and $(\ovfcal,\qcal)$ according to
$\{(\rho_{n-1,i} - \undrho_{n-1,i})\}_{i=1}^I$ and some fitness
$h$ for $(\undfcal,\qcal)$;

\item with probability $\sum_{i=1}^I
(\ovrho_{n-1,i}-\rho_{n-1,i})$, pick a fitness for the
$(\ovfcal,\qcal)$-chain according to $\{(\ovrho_{n-1,i} -
\rho_{n-1,i})\}_{i=1}^I$, pick some fitness $h$ for
$(\undfcal,\qcal)$, and pick a fitness for $(\fcal,\qcal)$
according to $\{(f_j M_{n,j})\}_{j> I}$;

\item note that there is no remaining probability mass since
$\sum_{i=1}^I \ovrho_{n-1,i}=1$.
\end{itemize}
This concludes the proof. It should be clear that the described coupling is
valid.
\end{proof}

\vspace{-0.5cm}\paragraph{Full Analysis.} Using our coupling idea we can also derive bounds on $N_{n,(j,k)}$,
defined as in Section~\ref{sec:full} to be the number of vertices
of fitness $f_j$ and degree $k$ at time $n$ in the
$(\fcal,\qcal)$-chain, in terms of the corresponding variables of
the $(\ovfcal,\qcal)$-chain and $(\undfcal,\qcal)$-chain. The coupling
has a similar flavor and its details are postponed to Section
\ref{sec:app:coupling_discrete} of the appendix.

%

\subsection{Fit-Get-Richer Phase}

Let $h = \sup_{j\geq 1} f_j < +\infty$, the case $h=+\infty$ being
treated in Section \ref{sec:unbounded}. Unlike the finite-type
case, when $J=+\infty$, we are not guaranteed that there exists a
solution of
\begin{equation}\label{eq:be1}
\sum_{j=1}^{J} \frac{f_j q_j}{\lambda - f_j} = 1,
\end{equation}
with $\lambda > h$. Observe, however, that in our proof of
Proposition~\ref{prop:alone} this was necessary for the existence
of a (summable) Perron-Frobenius eigenvector (see the expression
for $v_1$ in the proof of Proposition~\ref{prop:alone}). We will
actually show that the behavior of the process depends crucially
on the existence of such a solution. In this section, we consider
the case
\begin{equation}\label{eq:be2}
\sum_{j=1}^{J} \frac{f_j q_j}{h - f_j} > 1.
\end{equation}
We generalize Proposition \ref{prop:finitetype} exhibiting a
fit-get-richer behavior in this case. The following theorem
summarizes our result.
\begin{theorem}[Discrete Case: Fit-Get-Richer Phase]\label{thm:fitgetrich}
Let $1 \leq J \leq +\infty$, $h = \sup_{j\geq 1} f_j < +\infty$.
\begin{equation*}\label{eq:fitgetrich}
\text{Assume~~~~~~~~~~~~~~~~~~~~~~~~~~~~~~~~~~~~~~~~~~~~~}\sum_{j=1}^{J}
\frac{f_j q_j}{h - f_j}
>
1.~~~~~~~~~~~~~~~~~~~~~~~~~~~~~~~~~~~~~~~~~~~~~~~~~~~~~~~~~~
\end{equation*}
Then it holds that
\begin{enumerate}
\item for all $1 \leq j < J+1$, 
$\frac{M_{n,j}}{n} \to \nu_j,$ almost surely as $n \to +\infty$,

\item for all $1\leq j < J+1$ 
and $k \geq 1$, $\frac{N_{n,(j,k)}}{n} \to \eta_{(j,k)},$ almost
surely as $n \to +\infty$,
\end{enumerate}
where $\{\nu_j\}_j$ and $\{\eta_{(j,k)}\}_{j,k}$ are defined by
Equations~\eqref{eq:definition of
nu},~\eqref{eq:basic_fit_get_rich}.
\end{theorem}
%
%
%
%
%
%

\subsection{Innovation-Pays-Off Phase}\label{sec:boseeinstein}
\vspace{-0.1cm}Assume that $h = \sup_{j\geq 1} f_j < +\infty$ and that
\vspace{-0.3cm}\begin{equation}\label{eq:be3}
\ical(h)\equiv\sum_{j=1}^{J} \frac{f_j q_j}{h - f_j} \leq 1.
\end{equation}

\vspace{-0.3cm}\noindent It is easy to check that this is possible only if the fitness
supremum $h$ is not attained in $\fcal$ (see also the discussion in
Example~\ref{ex:fitgetrich} of the appendix). In particular, it
must be that $J =+\infty$.
Now set $\nu'_{j} = h \frac{q_j}{h - f_j}$, for $1
\leq j < +\infty$, and note in particular
that
\begin{equation}\label{eq:sum is not 2}
\sum_{j=1}^{+\infty} \nu'_j = \sum_{j=1}^{+\infty} (h-f_j)
\frac{q_j}{h - f_j} + \sum_{j=1}^{+\infty} f_j \frac{q_j}{h -
f_j}=1+\ical(h) \leq 2,~~~~~~~
\end{equation}
with equality only if there is equality in
(\ref{eq:be3})\footnote{Strictly speaking, the equality case
belongs to the fit-get-rich phase since Equation \eqref{eq:be1}
has a solution, namely $h$; nevertheless we include it in this
section because its proof is more similar to the
innovation-pays-off phase.}.
Also, for all $1\leq j < +\infty$ and $k \geq 1$, let
$\eta'_{(j,k)}$ be defined as $\eta'_{(j,k)} = \frac{h q_j}{k(h -
f_j)}\prod_{l=2}^{k}\frac{l}{l + h f_{j}^{-1}}$.
In particular, $\frac{\eta'_{(j,k+1)}}{\eta'_{(j,k)}} =
\frac{k}{k+1}\frac{k+1}{k+1 + h f_{j}^{-1}} = 1 - \frac{1+ h
f_{j}^{-1}}{k}(1+o(1))$, as $k$ gets large.
Hence, for fixed $j$, $\{\eta'_{(j,k)}\}_{k\geq 1}$ has tail
exponent $h f_j^{-1}$.
\begin{theorem}[Discrete Case: Innovation-Pays-Off Phase]\label{thm:boseeinstein}
Let $h = \sup_{j\geq 1} f_j < +\infty$. Assume
\begin{equation}\label{eq:boseeinstein}
\sum_{j=1}^{+\infty} \frac{f_j q_j}{h - f_j} \leq 1.
\end{equation}
Then it holds that
\begin{enumerate}
\item For all $1 \leq j < +\infty$,~~$ \frac{M_{n,j}}{n} \to
\nu'_j,$~~almost surely as $n \to +\infty$.

\item For all $1\leq j < +\infty$ and $k \geq 1$,~~
$\frac{N_{n,(j,k)}}{n} \to \eta'_{(j,k)},$~~almost surely as $n
\to +\infty$.
\end{enumerate}
\end{theorem}

\section{Open Problems}

A challenging open problem is to give an exact
quantitative description of the dynamics of the innovation-pays-off phase. Our
results imply that a constant fraction of the links ``escapes at infinity.''
But we know little about the transient behavior in this regime. How are the links distributed
among the highest fitnesses present in the system at any given time? At what rate are 
new nodes with higher fitnesses
taking over? How does the transient behavior depend on the fitness distribution?
This could have important practical implications. 



{
\section*{Acknowledgments}

We thank Kamal Jain for helpful discussions in the initial stages of this project.

}

\bibliographystyle{plain}
\bibliography{all}

\clearpage

\begin{appendix}
\section{Analysis of Bounded Discrete Fitness Distributions} \label{sec:app:discrete_finite}
\begin{prevproof}{Proposition}{prop:1fit}
We complete the proof of Proposition~\ref{prop:1fit} by computing
the largest positive eigenvalue $\lambda_1$ and the corresponding
left eigenvector $v_1$ of the matrix $A$. Because the $\xi_i$'s
are deterministic, it follows that $A_{ij} = \xi_{i,j}$ for all
$1\leq i,j\leq q$. To compute $\lambda_1$ we first compute the
corresponding {\em right} eigenvector. Note that
\begin{equation*}
\sum_{j=1}^q \xi_{i,j} = 2,
\end{equation*}
for all $1\leq i\leq q$ and therefore $u_1$ is $(1,\ldots,1)$ (up
to a constant factor) and $\lambda_1 = 2$. The left eigenvector
$v_1$ must satisfy,
\begin{equation*}
\sum_{i=1}^q (v_1)_i = 1,
\end{equation*}
by convention, as well as,
\begin{equation*}
\sum_{i=2}^q (v_1)_i = 2 (v_1)_1,
\end{equation*}
which with the previous equation implies $(v_1)_1 = 1/3$. Also,
for $2 \leq l\leq q-1$,
\begin{equation*}
l(v_1)_{l-1} -l (v_1)_l = 2 (v_1)_l,
\end{equation*}
or,
\begin{equation*}
\frac{(v_1)_l}{(v_1)_{l-1}} = \frac{l}{l+2}.
\end{equation*}
Therefore,
\begin{equation*}
(v_1)_k = \prod_{i=1}^k \frac{l}{l+2}.
\end{equation*}
Finally, by Theorem~\ref{thm:basic}, we get
\begin{equation*}
\frac{L_{n,k}}{n} = \frac{X_{n,k}}{kn} \to \frac{2(v_1)_k}{k} =
\mu_k.
\end{equation*}
almost surely as $n\to+\infty$.
\end{prevproof}

\medskip
\begin{prevproof}{Proposition}{prop:alone}
We complete the proof of Proposition~\ref{prop:alone} by computing
the largest positive eigenvalue $\lambda_1$ and the corresponding
left eigenvector $v_1$ of the matrix $A$ which has the following
form
\begin{equation*}
A_{ij} = f_i(q_j + \onebf_{\{i=j\}}).
\end{equation*}
We compute the corresponding $\lambda_1, v_1$. For all $1\leq
j\leq J$, $v_1$ must satisfy
\begin{equation}\label{eq:eigenvectoralone}
q_j \sum_{i=1}^J f_i (v_1)_i + f_j (v_1)_j = \lambda_1 (v_1)_j.
\end{equation}
By the convention
\begin{equation}\label{eq:conv}
a\cdot v_1 = 1
~~~~\Leftrightarrow~~~~ \sum_{i=1}^J f_i (v_1)_i =
1,
\end{equation}
it follows that, for all $1\leq j\leq
J$,
\begin{equation*}
(v_1)_j = \frac{q_j}{\lambda_1 - f_j}.
\end{equation*}
Plugging back into (\ref{eq:conv}), we get
\begin{equation*}
\sum_{j=1}^J \frac{f_j q_j}{\lambda_1 - f_j} = 1.
\end{equation*}
Therefore, $\lambda_1 = \lambda_0$ and $(v_1)_j = (\lambda_1)^{-1}
\nu_j$ for all $1\leq j\leq J$. The result follows by
Theorem~\ref{thm:basic}.
\end{prevproof}

\medskip
\begin{prevproof}{Proposition}{prop:finitetype}
Fix $1\leq j\leq J$ and $k\geq 1$. Set $r=k+1$ and $q=r J$.
Consider the following urn process which is a combination of those
in Propositions~\ref{prop:1fit} and~\ref{prop:alone}. We now have
a bin---indexed $(i,l)$---for each fitness $f_i$ and each degree
$l$ up to $k$. The number of balls in bin $(i,l)$ at time $n$ is
denoted $X_{n,(i,l)}$. The urn process is defined so that
$X_{n,(i,l)} = l N_{n,(i,l)}$ (see below). Also, for each $i$, the
bin $(i,r)$ counts all the links attached to a vertex of fitness
$f_i$ and degree more than $k$, that is we have
\begin{equation*}
X_{n,(i,r)} = \sum_{l\geq k+1} l N_{n,(i,l)}.
\end{equation*}
The activity of bin $(i,l)$ is $a_{(i,l)} = f_i$. Say at step $n$
we pick a ball from bin $(i,l)$ with $1 < l < r$. Then,
\begin{enumerate}
\item we choose a fitness, say $i'$, according to $\qcal$; \item
we add one ball to bin $(i',1)$; \item we remove $l$ balls from
bin $(i,l)$; \item we add $l+1$ balls to bin $(i,l+1)$.
\end{enumerate}
The cases $l=1, r$ are handled similarly (see
Proposition~\ref{prop:1fit}).

We compute matrix $A$. Let $(i,l)$ be such that $1 < l < r$. Then
row $(i,l)$ of $A$ is
\begin{equation*}
A_{(i,l),(i',l')} = \left\{\begin{array}{ll}
- f_i l, & i' = i, l'=l\\
f_i (l+1), & i' = i, l'=l+1\\
f_i q_{i'}, & l' = 1\\
0, & \mathrm{o.w.}
\end{array}
\right.
\end{equation*}
For $l=1$, we get
\begin{equation*}
A_{(i,1),(i',l')} = \left\{\begin{array}{ll}
f_i (-1 + q_i), & i' = i, l'=1\\
2 f_i , & i' = i, l'=2\\
f_i q_{i'}, & i'\neq i, l' = 1\\
0, & \mathrm{o.w.}
\end{array}
\right.
\end{equation*}
and, for $l=r$,
\begin{equation*}
A_{(i,r),(i',l')} = \left\{\begin{array}{ll}
f_i, & i' = i, l'=r\\
f_i q_{i'}, & l' = 1\\
0, & \mathrm{o.w.}
\end{array}
\right.
\end{equation*}
We compute the corresponding $\lambda_1, u_1, v_1$. Consider the
following guess for $u_1$
\begin{equation*}
(u_1)_{(i,l)} = \frac{f_i}{\lambda_0 - f_i},
\end{equation*}
for all $1\leq i\leq J$ and $1\leq l\leq r$ where $\lambda_0$ is
defined in (\ref{eq:befinite}). Then we have for $1\leq i\leq J$
and $1 \leq l \leq q$,
\begin{eqnarray*}
\sum_{(i',l')} A_{(i,l),(i',l')} (u_1)_{(i',l')}
&=& f_i \sum_{i'} \frac{f_{i'} q_{i'}}{\lambda_0 - f_{i'}} + \frac{f_i^2}{\lambda_0 - f_i}\\
&=& f_i + \frac{f_i^2}{\lambda_0 - f_i}\\
&=& \frac{f_i}{\lambda_0 - f_i} \left(\lambda_0 - f_i + f_i\right)\\
&=& \lambda_0 (u_1)_{(i,l)},
\end{eqnarray*}
where we used (\ref{eq:befinite}). Hence, the Perron-Frobenius
eigenvalue is $\lambda_1 = \lambda_0$ and the corresponding right
eigenvector is $u_1$ as above.

It remains to compute $v_1$. Define the auxiliary vector
\begin{equation*}
(\tilv_1)_i = \sum_{l=1}^r (v_1)_{(i,l)},
\end{equation*}
for $1\leq i\leq J$. Then, by looking at column $(i,1)$ of $A$, we
must have
\begin{equation}\label{eq:finite1}
q_{i} \sum_{i'=1}^J f_{i'} (\tilv_1)_{i'} - f_i (v_1)_{(i,1)} =
\lambda_1 (v_1)_{(i,1)},
\end{equation}
for all $1\leq i\leq J$. From column $(i,r)$ we get
\begin{equation}\label{eq:finite2}
f_i (r(v_1)_{(i,r-1)} + (v_1)_{i,r}) = \lambda_1 (v_1)_{(i,r)}.
\end{equation}
Finally, for $1< l< r$, column $(i,l)$ gives
\begin{equation}\label{eq:finite3}
f_i (l(v_1)_{(i,l-1)} - l(v_1)_{i,l}) = \lambda_1 (v_1)_{(i,l)}.
\end{equation}
Summing~(\ref{eq:finite1}),~(\ref{eq:finite2}),
and~(\ref{eq:finite3}), we obtain
\begin{equation*}
q_{i} \sum_{i'=1}^J f_{i'} (\tilv_1)_{i'} + f_i (\tilv_1)_{i} =
\lambda_1 (\tilv_1)_{i}.
\end{equation*}
This is identical to (\ref{eq:eigenvectoralone}) from
Proposition~\ref{prop:alone} and therefore
\begin{equation*}
(\tilv_1)_i = \frac{q_i}{\lambda_1 - f_i},
\end{equation*}
for all $1\leq i\leq J$. Also, from (\ref{eq:finite3}), for $1< l<
r$, we get
\begin{equation*}
\frac{(v_1)_{(i,l)}}{(v_1)_{(i,l-1)}} = \frac{l}{l+\lambda_1
f_{i}^{-1}}.
\end{equation*}
By our convention,
\begin{equation*}
\sum_{i'=1}^J f_{i'} (\tilv_1)_{i'} = 1,
\end{equation*}
we get from (\ref{eq:finite1}),
\begin{equation*}
(v_1)_{(i,1)} = \frac{q_i}{\lambda_1 + f_i}.
\end{equation*}

From Theorem~\ref{thm:basic}, we derive
\begin{equation*}
\frac{N_{n,(j,k)}}{n} = \frac{X_{n,(j,k)}}{kn} \to \frac{\lambda_1
(v_1)_{(j,k)}}{k} = \eta_{(j,k)},
\end{equation*}
almost surely as $n\to +\infty$. This concludes the proof.
\end{prevproof}

\section{Analysis of Countable Discrete Fitness Distributions} \label{sec:app:discrete_countable}

\subsection{Coupling} \label{sec:app:coupling_discrete}
We derive bounds on $N_{n,(j,k)}$, defined to be the number of
vertices of fitness $f_j$ and degree $k$ at time $n$ in the
$(\fcal,\qcal)$-chain of Section \ref{sec:countable}. Fix $I$ and
let $\ovnn_{n,(j,k)}, \undnn_{n,(j,k)}$ be the corresponding
variables for the chains $(\ovfcal,\qcal)$ and $(\undfcal,\qcal)$
of Section \ref{sec:coupling_discrete} defined by the
$I$-truncations of $\fcal$. Since the latter have finite fitness
distributions, we can apply Proposition~\ref{prop:finitetype}. Let
$\oveta_{(j,k)}$ and $\undeta_{(j,k)}$ be the almost sure limits
of $n^{-1}\ovnn_{n,(j,k)}$ and $n^{-1}\undnn_{n,(j,k)}$. For the
full coupling, we also need the degree tails for a fixed fitness.
Let
\begin{equation*}
T_{n,(j,k)} = \sum_{k' \geq k} k' N_{n,(j,k')},
\end{equation*}
and similarly for $\ovtt_{n,(j,k)}$ and $\undtt_{n,(j,k)}$. Also,
let
\begin{equation*}
\ovtau_{(j,k)} = \sum_{k'\geq k} k' \oveta_{(j,k')},
\end{equation*}
and similarly for $\undtau_{(j,k)}$. These are well-defined
because the partial sums are increasing and bounded by 2 (see the
proof of Proposition~\ref{prop:finitetype}). The following lemma
provides a full coupling of the chains $(\fcal,\qcal)$,
$(\ovfcal,\qcal)$ and $(\undfcal,\qcal)$.

\begin{lemma}[Coupling: Full Analysis]\label{lem:couplingfull}
For all $1 \leq j \leq I$ and $k\geq 1$, it holds almost surely
that
\begin{equation*}
\limsup_{n \to +\infty} \frac{T_{n,(j,k)}}{n} \leq \ovtau_{(j,k)}
\text{~~~~and~~~~} \liminf_{n \to +\infty} \frac{T_{n,(j,k)}}{n}
\geq \undtau_{(j,k)}.
\end{equation*}
\end{lemma}

\begin{prevproof}{Lemma}{lem:couplingfull}
As in Lemma~\ref{lem:couplingalone}, we couple the
$(\fcal,\qcal)$-chain and the truncations. We use the notation of
Lemma~\ref{lem:couplingalone}. Also, for $k\geq 1$, let $D_n$ be
the degree of the vertex picked at time $n$ in the
$(\fcal,\qcal)$-chain (and similarly for $\ovdd_n, \unddd_n$). For
$1\leq i\leq I$ and $k \geq 1$, let $\sigma_{n-1,(i,k)}$ be the
probability of the event $\{F_n = f_i, D_n \geq k\}$ given the
state after time $n-1$ in the $(\fcal,\qcal)$-chain (and similarly
for $\ovsigma_n, \undsigma_n$). We require the following
conditions to be satisfied:
\begin{enumerate}
\item For all $n \geq 1$,
\begin{equation*}
\ovff_n \leq F_n \leq \undff_n,
\end{equation*}
\begin{equation*}
\ovff_n' \leq F_n' \leq \undff_n'.
\end{equation*}

\item For all $n\geq 1$ and all $1\leq i\leq I$,
\begin{equation*}
\undmm_{n,i} \leq M_{n,i} \leq \ovmm_{n,i}.
\end{equation*}

\item For all $n\geq 1$ and all $1\leq i\leq I$,
\begin{equation*}
\undrho_{n,i} \leq \rho_{n,i} \leq \ovrho_{n,i}.
\end{equation*}

\item For all $n\geq 1$, $1\leq i\leq I$, and $k\geq 1$,
\begin{equation*}
\undtt_{n,(i,k)} \leq T_{n,(i,k)} \leq \ovtt_{n,(i,k)}.
\end{equation*}

\item For all $n\geq 1$, $1\leq i\leq I$, and $k\geq 1$,
\begin{equation*}
\undsigma_{n,(i,k)} \leq \sigma_{n,(i,k)} \leq \ovsigma_{n,(i,k)}.
\end{equation*}
\end{enumerate}
These conditions are somewhat redundant but we keep all of them
for clarity. In particular, note that 3. follows from 1.~and 2.,
that 5.~follows from 1.~and 4.,
 and that 2.~and 3.~are special cases of 4.~and 5.
Assume these conditions hold up to $n-1$. Our step-by-step
coupling has two parts. First, we pick the fitnesses $\ovff_n,
F_n, \undff_n, \ovff_n', F_n', \undff_n'$ using the scheme
described in the proof of Lemma~\ref{lem:couplingalone}. We then
pick the degrees $\unddd_n, D_n, \ovdd_n$ by picking a {\em
single} uniform random variable in $[0,1]$ and ``inverting''
simultaneously the tails $\{\undsigma_{n,(\undff_n,k)}\}_{k \geq
1}$, $\{\sigma_{n,(F_n,k)}\}_{k \geq 1}$, and
$\{\ovsigma_{n,(\ovff_n,k)}\}_{k \geq 1}$. (This is sometimes
called the ``inverse transform sampling method''.) It is easy to
check that all conditions are then satisfied at time $n$.
\end{prevproof}

\subsection{Fit-Get-Richer Phase}

\begin{prevproof}{Theorem}{thm:fitgetrich}
We only need to consider the case $J=+\infty$. Fix $1 \leq j <
+\infty$ and $k \geq 1$. Let $1\leq I < +\infty$ and consider once
again the $I$-truncations of the $(\fcal,\qcal)$-chain. Let
$\undnu^I_j, \ovnu_j^I, \undeta^I_{(j,k)}, \oveta^I_{(j,k)},
\undtau^I_{(j,k)}, \ovtau^I_{(j,k)}$ be as in
Lemmas~\ref{lem:couplingalone},~\ref{lem:couplingfull} (we now
indicate the dependence on $I$ because we will need to take $I \to
+\infty$). Similarly, let $\undlambda^I_0$ and $\ovlambda^I_0$ be
the largest solution to (\ref{eq:be1}) for the lower and upper
truncations. By the coupling lemmas, it suffices to prove
\begin{equation}\label{eq:lambda0}
\undlambda^I_0, \ovlambda^I_0 \to \lambda_0,
\end{equation}
as $I \to +\infty$. Indeed, in that case
\begin{equation*}
\undnu^I_j, \ovnu^I_j \to \nu_j,
\end{equation*}
as $I \to +\infty$, which implies
\begin{equation*}
\frac{M_{n,j}}{n} \to \nu_j,
\end{equation*}
by Lemma~\ref{lem:couplingalone}. Also, for all $l \leq k$,
\begin{equation*}
\undeta^I_{(j,l)}, \oveta^I_{(j,l)} \to \eta_{(j,l)},
\end{equation*}
as $I \to +\infty$, which implies
\begin{equation*}
\undtau^I_{(j,k)} = \undnu^I_j - \sum_{l\leq k} l\undeta^I_{(j,l)}
\to \nu_j - \sum_{l\leq k} l\eta_{(j,l)},
\end{equation*}
as $I \to +\infty$, and similarly for $\ovtau^I_{(j,k)}$. This
also holds for $k-1$ so that, by Lemma~\ref{lem:couplingfull}, we
have
\begin{equation*}
\frac{N_{n,(j,k)}}{n} \to \eta_{(j,k)},
\end{equation*}
almost surely as $n \to +\infty$.

It remains to prove (\ref{eq:lambda0}). We argue about
$\ovlambda^I_0$. The proof for $\undlambda^I_0$ is similar and is
omitted. Let
\begin{equation*}
S(\lambda) := \sum_{i=1}^{+\infty} \frac{f_i q_i}{\lambda - f_i},
\qquad \ovss^I(\lambda) := \sum_{i=1}^{+\infty} \frac{\ovf_i
q_i}{\lambda - \ovf_i}.
\end{equation*}
Note that for $\lambda' > \lambda > h$, we have
\begin{equation*}
S(\lambda'), S(\lambda) \leq \frac{h}{\lambda - h}, \qquad
S(\lambda') < S(\lambda),
\end{equation*}
and
\begin{equation*}
\left|S(\lambda') - S(\lambda)\right| \leq \frac{|\lambda' -
\lambda|h}{|\lambda' - h||\lambda - h|}.
\end{equation*}
Therefore, $S$ is continuous and strictly decreasing on $\{\lambda
> h\}$. Also, by definition of $\ovfcal$, we have
\begin{equation*}
\ovrr^I(\lambda) := S(\lambda) - \ovss^I(\lambda) =
\sum_{i=I+1}^{+\infty} \frac{f_i q_i}{\lambda - f_i}.
\end{equation*}
Therefore, for $\lambda > h$,
\begin{equation*}
\left|\ovrr^I(\lambda)\right| \leq \frac{h}{\lambda - h}
\sum_{i=I+1}^{+\infty} q_i \to 0,
\end{equation*}
as $I \to +\infty$. Hence, for all $\eps > 0$ (small enough),
\begin{equation*}
\lim_{I \to \infty} \ovss^I(\lambda_0 +\eps) = S(\lambda_0 + \eps)
< 1, \qquad \lim_{I \to \infty} \ovss^I(\lambda_0 - \eps) =
S(\lambda_0 - \eps) > 1,
\end{equation*}
so that eventually
\begin{equation*}
\lambda_0 - \eps \leq \ovlambda^I_0 \leq \lambda_0 + \eps.
\end{equation*}
Since $\eps > 0$ is arbitrary, we have (\ref{eq:lambda0}).
\end{prevproof}

\subsection{Innovation-Pays-Off Phase}

\begin{example}\label{ex:fitgetrich}
The case $J < +\infty$ always satisfies (\ref{eq:be2}). Indeed, in that case,
\begin{equation}\label{eq:infinity}
\sum_{j=1}^{J} \frac{f_j q_j}{h - f_j} = +\infty.
\end{equation}
Likewise, when
$J=+\infty$ and the fitness supremum $h$ is attained, we also get
(\ref{eq:infinity}).
\end{example}

\begin{example}
Consider the case $f_j = 1 - j^{-1}$ for all $j \geq 1$ and
\begin{equation*}
q_j = \frac{j^{2+\theta}}{\zeta(2+\theta)},
\end{equation*}
where $\zeta$ is the Riemann zeta function. In particular, by
definition, $\sum_{j\geq 1} q_j = 1$. Here, $h = 1$ is not
attained. We now compute the sum in (\ref{eq:be3}). We have
\begin{eqnarray*}
\sum_{j=1}^{J} \frac{f_j q_j}{h - f_j}
&=& \sum_{j\geq 1} \frac{(1 - j^{-1})\zeta^{-1}(2+\theta)j^{-2-\theta}}{j^{-1}}\\
&=& \zeta^{-1}(2+\theta) \left(\sum_{j\geq 1}j^{-1-\theta} - \sum_{j\geq 1} j^{-2-\theta}\right)\\
&=& \frac{\zeta(1+\theta) - \zeta(2+\theta)}{\zeta(2+\theta)}.
\end{eqnarray*}
One can check that the last line is $<1$ when $\theta > 1$. This
example can be seen as a ``discretization'' of the example given
in~\cite{BianconiBarabasi:01}.
\end{example}

\begin{prevproof}{Theorem}{thm:boseeinstein}
We use the notations of Theorem~\ref{thm:fitgetrich}. Similarly to
Theorem~\ref{thm:fitgetrich}, it suffices to prove
\begin{equation}\label{eq:lambda0be}
\undlambda^I_0, \ovlambda^I_0 \to h,
\end{equation}
as $I \to +\infty$. Let
\begin{equation*}
h^I = \sup_{j\leq I} f_j.
\end{equation*}
By a remark above the statement of the Theorem, we know that $h^I
< h$ and $h^I \to h$ as $I \to +\infty$.

We first argue about $\ovlambda^I_0$. Note that $\ovlambda^I_0 >
h^I$. Also, $\ovss^I(h) < S(h) \leq 1$ and therefore
$\ovlambda^I_0 \leq h$. That implies $\ovlambda^I_0 \to h$.

Now consider the case of $\undlambda^I_0$. Let
\begin{equation*}
\undrr^I(\lambda) := S(\lambda) - \undss^I(\lambda).
\end{equation*}
We have, for all $\eps > 0$,
\begin{equation*}
S(h + \eps) < S(h) \leq 1,
\end{equation*}
and
\begin{equation*}
\left|\undrr^I(h + \eps)\right| \leq \frac{h}{\eps}
\sum_{i=I+1}^{+\infty} q_i \to 0,
\end{equation*}
as $I \to +\infty$. Hence, for all $\eps > 0$,
\begin{equation*}
\lim_{I \to \infty} \undss^I(h +\eps) = S(h + \eps) < 1,
\end{equation*}
so that eventually
\begin{equation*}
h^I < \undlambda^I_0 \leq h + \eps.
\end{equation*}
Since $\eps > 0$ is arbitrary, we have $\undlambda^I_0 \to h$ as
$I\to +\infty$.
\end{prevproof}

\subsection{Unbounded Countable Case}\label{sec:unbounded}

Assume $h = \sup_{j\geq 1} f_j = +\infty$, i.e.~the set of
fitnesses is unbounded. In that case, the lower bounds in the
coupling lemmas cannot be used but it turns out that the upper
bounds suffice to characterize the limit behavior of the process.
\begin{theorem}[Discrete Case: Unbounded Fitness]\label{thm:unbounded}
Assume $\sup_{j\geq 1} f_j = +\infty$. Then it holds that
\begin{enumerate}
\item For all $1 \leq j < +\infty$,
\begin{equation*}
\frac{M_{n,j}}{n} \to q_j,
\end{equation*}
almost surely as $n \to +\infty$.

\item For all $1\leq j < +\infty$ and $k > 1$,
\begin{equation*}
\frac{T_{n,(j,k)}}{n} \to 0,
\end{equation*}
almost surely as $n \to +\infty$.
\end{enumerate}
\end{theorem}
\begin{proof}
Fix $1 \leq j < +\infty$ and $k > 1$. We use the upper bounds in
the coupling Lemmas~\ref{lem:couplingalone}
and~\ref{lem:couplingfull}. We use the notations of
Theorems~\ref{thm:fitgetrich} and~\ref{thm:boseeinstein}. We have
that $h^I \to +\infty$ and therefore $\ovlambda^I_0 > h^I \to
+\infty$. Therefore, plugging into the equations for $\ovnu^I_j$
and $\ovtau^I_{(j,k)} = \ovnu^I_j - \sum_{l\leq k}
l\oveta^I_{(j,l)}$, we get
\begin{equation*}
\limsup_{n\to\infty}\frac{M_{n,j}}{n} \leq q_j,
\end{equation*}
and
\begin{equation*}
\limsup_{n\to \infty}\frac{T_{n,(j,k)}}{n} \leq 0,
\end{equation*}
almost surely. We get 2.~immediately. To get 1., consider the
following chain $\{X_{n,i}\}_{n,i \geq 0}$. Pick a fitness say
$F_0$ according to $\qcal$ and let $X_0 = e_{F_0}$. Then at each
time step, pick a fitness $F_n$ according to $\qcal$ and set $X_n
= X_{n-1} + e_{F_n}$. This chain can clearly be coupled with the
$(\fcal,\qcal)$-chain in such a way that $M_{n} \geq X_n$ for all
$n$. Now it is easy to see that $X_{n,j} \to q_j$ as $n\to
+\infty$, and therefore
\begin{equation*}
\liminf_{n\to \infty} \frac{M_{n,j}}{n} \geq q_j.
\end{equation*}
This concludes the proof.
\end{proof}

\section{Analysis of Continuous Fitness Distributions} \label{sec:app:continuous}


In this section, we analyze the preferential attachment scheme
under continuous fitness
distributions. Let $h < +\infty$---the unbounded case is
treated in Appendix \ref{sec:app:continuous_unbounded}---and
let $g:[0,h] \to \real_+$ be a continuous density function. 
Consider the preferential attachment process 
with $\fcal = [0,h]$ and $\qcal$
the distribution defined by $g$. The dynamical behavior parallels
the one observed in the discrete case, namely
\begin{enumerate}
\item the fit-get-richer scenario taking place when $ \int_{0}^h
\frac{x g(x)}{h - x}\diff x \geq 1,$

\item the innovation-pays-off scenario taking place when $\int_{0}^h
\frac{x g(x)}{h - x}\diff x < 1.$
\end{enumerate} 
The analysis requires a more sophisticated coupling 
argument than that for the discrete case described
in Section~\ref{sec:countable}.

\subsection{Coupling}

We discretize the $(\fcal, \qcal)$-chain in the following way,
which lets us bound the relevant quantities from below only. It
will turn out that the lower bound is sufficient for our purposes.
Fix $1 < I < +\infty$, an integer with $\eps = h\frac{1}{I}$. For
$1 \leq i \leq I$, let
\begin{equation*}
\ovf_i = h\frac{i}{I},
\end{equation*}
\begin{equation*}
\undf_i = h\frac{i-1}{I},
\end{equation*}
and
\begin{equation*}
\tildeq_i = \int_{\undf_i}^{\ovf_i} g(x)\diff x.
\end{equation*}
Denote $\tildeqcal$ the distribution over $\{1,2,\ldots,I\}$
defined by $\{\tildeq_i\}_{i=1}^I$. For reasons that will be clear
in Section~\ref{sec:becont}, we allow $\int_0^h g(x)\diff x < 1$.
Consider the following finite balls-in-bins process with $q = I+1$
bins. The activities are
\begin{equation*}
a_{i} = \left\{
\begin{array}{ll}
\ovf_i, & \mathrm{if\ } i\leq I,\\
h, & \mathrm{if\ }i = I+1.
\end{array}
\right.
\end{equation*}
For the initial load, let $1 \leq i^*\leq I$ be picked according
to $\tildeqcal$ and pose
\begin{equation*}
X_{0,i} = \left\{
\begin{array}{ll}
2, & \mathrm{if\ }i = i^*,\\
0, & \mathrm{o.w.}
\end{array}
\right.
\end{equation*}
The update vectors are defined as follows for $1 \leq i \leq I$:
pick $i^*$ according to $\tildeqcal$ (set $i^* = +\infty$ with
probability $1 - G$ where $G = \int_0^h g(x) \diff x$), let
$\gamma_i = 1$ with probability $\undf_i/\ovf_i$ and $0$ o.w., and
set
\begin{equation*}
\xi_{i,i'} = \gamma_i \ind_{\{i' = i\}} + \left\{
\begin{array}{ll}
1, & \rmif i'=i^*,\\
1, & \rmif i'=I+1, \gamma_i = 0,\\
0, & \rmow
\end{array}
\right.
\end{equation*}
and
\begin{equation*}
\xi_{I+1,i'} = \left\{
\begin{array}{ll}
1, & \rmif i'=i^*,\\
1, & \rmif i'=I+1,\\
0, & \rmow
\end{array}
\right.
\end{equation*}
Because this chain is not exactly of the type described in
Section~\ref{sec:finitetype}, we cannot appeal directly to
Proposition~\ref{prop:alone}. Therefore, we give a separate
analysis here. Let $\tildelambda_0 > h - \eps$ be a solution to
\begin{equation}\label{eq:bediscr}
\sum_{j=1}^{I} \frac{\ovf_j \tildeq_j}{\tildelambda_0 - \undf_j} +
h \frac{(1+G)\tildelambda_0^{-1}\eps} {\tildelambda_0 - (h -
\eps)}= 1.
\end{equation}
By monotonicity, it is clear that there is a unique such solution.
For $1 \leq j \leq I$, let
\begin{equation*}
\tildenu_{j} = \tildelambda_0 \frac{\tildeq_j}{\tildelambda_0 -
\undf_j},
\end{equation*}
and
\begin{equation*}
\tildenu_{I+1} = \tildelambda_0
\frac{(1+G)\tildelambda_0^{-1}\eps}{\tildelambda_0 - (h - \eps)}.
\end{equation*}
Note that
\begin{eqnarray*}
\left(1+\frac{\eps}{\lambda_0}\right)\sum_{j=1}^{I+1} \tildenu_j
&=& \sum_{j=1}^{I} \tildeq_j + (1+G)\eps + \sum_{j=1}^{I}
\frac{\ovf_j \tildeq_j}{\tildelambda_0 - \undf_j} +
h \frac{(1+G)\tildelambda_0^{-1}\eps}{\tildelambda_0 - (h - \eps)}\nonumber\\
&=& G + (1+G)\frac{\eps}{\lambda_0} + 1\nonumber\\
&=& (1+G)\left(1+\frac{\eps}{\lambda_0}\right),
\end{eqnarray*}
so that
\begin{equation}\label{eq:sum1g}
\sum_{j=1}^{I+1} \tildenu_j = 1+G.
\end{equation}
We prove the following.
\begin{lemma}[Discretization]\label{lemma:discretization}
For all $1 \leq j\leq I+1$,
\begin{equation*}
\frac{X_{n,j}}{n} \to \tildenu_j,
\end{equation*}
almost surely as $n \to +\infty$.
\end{lemma}
\begin{proof}
The matrix $A$ has the following form: for $1 \leq i \leq I$, $1
\leq j \leq I + 1$,
\begin{equation*}
A_{ij} = \ovf_i\left(\tildeq_j\ind_{\{j\leq I\}} +
\frac{\undf_i}{\ovf_i}\onebf_{\{j=i\}} +
\frac{\eps}{\ovf_i}\onebf_{\{j = I+1\}}\right),
\end{equation*}
and for $i=I+1$
\begin{equation*}
A_{I+1,j} = h\left(\tildeq_j\ind_{\{j\leq I\}} + \onebf_{\{j =
I+1\}}\right).
\end{equation*}

We compute the corresponding $\lambda_1, v_1$. Note that by
Theorem~\ref{thm:basic} and the law of large numbers, it is clear
that
\begin{equation}\label{eq:sum1g2}
\sum_{i=1}^{I+1} \lambda_1 (v_1)_i = 1+G.
\end{equation}
For all $1\leq j\leq I$, $v_1$ must satisfy
\begin{equation*}
\tildeq_j \sum_{i=1}^{q} a_i (v_1)_i +
\ovf_j\frac{\undf_j}{\ovf_j} (v_1)_j = \lambda_1 (v_1)_j.
\end{equation*}
By the convention
\begin{equation}\label{eq:convcont}
\sum_{i=1}^q a_i (v_1)_i = 1,
\end{equation}
it follows that for all $1\leq j\leq I$
\begin{equation*}
(v_1)_j = \frac{\tildeq_j}{\lambda_1 - \undf_j}.
\end{equation*}
Also for $i = I+1$, we must have
\begin{equation*}
\sum_{i=1}^{I} \ovf_i \frac{\eps}{\ovf_i} (v_1)_i + h (v_1)_{I+1}
= \eps \left((1+G)\lambda_1^{-1} - (v_1)_{I+1}\right) + h
(v_1)_{I+1} = \lambda_1 (v_1)_{I+1},
\end{equation*}
where we have used (\ref{eq:sum1g2}). Therefore,
\begin{equation*}
(v_1)_{I+1} = \frac{(1+G)\lambda_1^{-1}\eps} {\lambda_1 - (h -
\eps)}.
\end{equation*}
Plugging back into (\ref{eq:convcont}), we get
\begin{equation*}
\sum_{j=1}^{I} \frac{\ovf_j \tildeq_j}{\lambda_1 - \undf_j} + h
\frac{(1+G)\lambda_1^{-1}\eps} {\lambda_1 - (h - \eps)}= 1.
\end{equation*}
Therefore, $\lambda_1 = \tildelambda_0$ and $(v_1)_j =
(\lambda_1)^{-1} \tildenu_j$ for all $1\leq j\leq q$. The result
follows by Theorem~\ref{thm:basic}.
\end{proof}

Consider again the $(\fcal,\qcal)$-chain. For $n \geq 0$ and
$1\leq j \leq I$, let $M_{n,j}$ be the number of edges with an
endpoint of fitness in $(\undf_j,\ovf_j)$ (counting twice edges
with two endpoints of fitness in $(\undf_j,\ovf_j)$). Then we have
the following.
\begin{lemma}[Coupling: Continuous Case]\label{lem:couplingalonecont}
For all $1 \leq j \leq I$, it holds that
\begin{equation*}
\liminf_{n \to +\infty} \frac{M_{n,j}}{n} \geq \tildenu_j,
\end{equation*}
almost surely.
\end{lemma}
\begin{proof}
This proof is similar to the proof of
Lemma~\ref{lem:couplingalone}. Consider the $(\fcal,
\qcal)$-chain. At step $n \geq 1$, we first pick a vertex
according to weighted preferential attachment. Let $F_n$ be the
fitness of the chosen vertex, and denote $\rho_{n-1,i}$ the
probability that $F_n \in (\undf_i,\ovf_i)$ given the state after
time $n-1$. Secondly, we add a new vertex with fitness according
to $\qcal$. Let $F_n'$ be the fitness of this new vertex.
Similarly for the discretized chain, we first pick a bin $i$ by
weighted preferential attachment and then an $i^*$ according to
$\tildeqcal$. We also pick $\gamma_i$ a
Bernoulli($\undf_i/\ovf_i$). We let
\begin{equation*}
\undff_n = \left\{
\begin{array}{ll}
\ovf_{i}, & \rmif \gamma_i = 1,\\
h, & \rmif \gamma_i = 0.
\end{array}
\right. \text{~~and~~}\undff_n' = \left\{
\begin{array}{ll}
\ovf_{i^*}, & \rmif i^* \leq I,\\
+\infty, & \rmif i^* = +\infty,
\end{array}
\right.
\end{equation*}
We denote $\undrho_{n-1,i}$ the probability that $\undff_n =
\ovf_i$ and $\gamma_i = 1$ given the state after time $n-1$
\footnote{ The specification that $\gamma_i = 1$ is relevant only
in the case $i = I$.}. We couple the two chains so as to preserve
the following conditions:
\begin{enumerate}
\item For all $n \geq 1$,
\begin{equation*}
F_n \leq \undff_n,
\end{equation*}
and
\begin{equation*}
F_n' \leq \undff_n'.
\end{equation*}

\item For all $n\geq 1$ and all $1\leq i\leq I$,
\begin{equation*}
\undmm_{n,i} \leq M_{n,i}.
\end{equation*}

\item For all $n\geq 1$ and all $1\leq i\leq I$,
\begin{equation*}
\undrho_{n,i} \leq \rho_{n,i}.
\end{equation*}
\end{enumerate}

{
Note that 3.~follows easily from 1., 2.~and the definition of
$\gamma_i$. In fact, the reason for using the ``rejection'' 
variable $\gamma_i$ is to keep $\undrho_{n,i}$ small by making its 
numerator small---with a contribution of only $\undf_i$---while 
preserving a large denominator.} 
Here is how our coupling works. In the initial
configuration, the $(\fcal,\qcal)$-chain has one vertex with a
self-loop and fitness $F_0=f_i$, where $f_i$ is picked according
to $\qcal$; the discretized chain can be coupled so that two balls
are added to a bin with activity $\undff_0 = \ovf_i$ with
probability $\undf_i/\ovf_i$ and $\undff_0=h$ with probability
$1-\undf_i/\ovf_i$. Therefore the conditions are satisfied at time
$0$ by construction. Assume Conditions 1., 2., and 3. are
satisfied at time $n-1$; we will show then that they are also
satisfied at time $n$. First, consider picking fitness for the new
vertex. In the $(\fcal,\qcal)$-chain, $F'_n = f_i$, where $f_i$ is
picked according to $\qcal$; the choice of the discretized chain
can be coupled so that $\undff'_n = \ovf_i$. Therefore, $ F_n'
\leq \undff_n'. $ Now consider the step of choosing an old vertex.
By 3., it is clear how to choose the $F$'s so as to satisfy 1.~and
2. Indeed, proceed as follows:
\begin{itemize}
\item With probability $\sum_{i=1}^I \undrho_{n-1,i}$, pick a bin
according to $\{{\undrho_{{n-1},i}}\}_{i=1}^I$ in the discretized
chain, say $i$, and pick a fitness according to weighted
preferential attachment restricted to $(\undf_i, \ovf_i)$ for the
$(\fcal,\qcal)$-chain (the interval $(\undf_i, \ovf_i)$ is
nonempty by 2.); \item With remaining probability, pick bin $I+1$
for the discretized chain, pick an interval according to
$\{(\rho_{{n-1},i} - \undrho_{{n-1},i})\}_{i=1}^I$, say $(\undf_i,
\ovf_i)$, and pick a fitness according to weighted preferential
attachment restricted to $(\undf_i, \ovf_i)$ for the
$(\fcal,\qcal)$-chain.
\end{itemize}
This concludes the proof.
\end{proof}

\subsection{Fit-Get-Richer Phase}

Assume the density $g$ is defined on $[0,h]$ with $h < +\infty$
and assume further that $g(x) > 0$ for all $x \in (0,h)$ (we allow
$0$ at the endpoints). In this section, we consider the case
\begin{equation}\label{eq:be2cont}
\int_{0}^h \frac{x g(x)}{h - x}\diff x \geq 1.
\end{equation}
The remaining cases are treated in the following two subsections.
\begin{example}\label{ex:fitgetrichcont}
An important special case of (\ref{eq:be2cont}) is when $g(h) >
0$. Indeed, take any $\delta > 0$ small and let $\delta' =
\inf_{x\in [h-\delta,h]} g(x)$. Note that $\delta' > 0$ by
assumption. Then,
\begin{eqnarray*}
\int_{0}^h \frac{x g(x)}{h - x}\diff x
&\geq& \int_{h-\delta}^h \frac{x g(x)}{h - x}\diff x\\
&\geq& (h - \delta)\delta' \int_{h-\delta}^h \frac{1}{h - x}\diff x\\
&\geq& (h - \delta)\delta' \int_{0}^{\delta} \frac{1}{y}\diff y\\
&=& +\infty\\
&\geq& 1.
\end{eqnarray*}
This example will turn out to be useful in
Section~\ref{sec:becont}.
\end{example}

By (\ref{eq:be2cont}) and monotonicity, there exists a solution
$\lambda_0 \geq h$ to
\begin{equation}\label{eq:be1cont}
\int_{0}^h \frac{x g(x)}{\lambda_0 - x}\diff x = 1.
\end{equation}
For $0 \leq a < b \leq h$, let
\begin{equation*}
\nu_{[a,b]} = \lambda_0 \int_{a}^b \frac{g(x)}{\lambda_0 - x}
\diff x.
\end{equation*}
Note in particular that
\begin{equation}\label{eq:sum1g3}
\nu_{[0,h]} = \int_{0}^h (\lambda_0 - x) \frac{g(x)}{\lambda_0 -
x} \diff x + \int_{0}^h \frac{x g(x)}{\lambda_0 - x} \diff x = 1 +
G,
\end{equation}
as one would expect (but see Section~\ref{sec:becont} below).
Also, for $n \geq 0$, let $M_{n,[a,b]}$ be the number of edges
with an endpoint of fitness in $[a,b]$ (counting twice edges with
two endpoints of fitness in $[a,b]$).

We prove the following.
\begin{theorem}[Continuous Case: Fit-Get-Richer Phase]\label{thm:fitgetrichcont}
Assume $g$ is defined on $[0,h]$ with $h < +\infty$ and assume
further that $g(x) > 0$ for all $x \in (0,h)$ and
\begin{equation*}
\int_{0}^h \frac{x g(x)}{h - x}\diff x \geq 1.
\end{equation*}
Then it holds that for all $0 \leq a < b \leq h$,
\begin{equation*}
\frac{M_{n,[a,b]}}{n} \to \nu_{[a,b]},
\end{equation*}
almost surely as $n \to +\infty$.
\end{theorem}
\begin{proof}
Note that the law of large numbers implies
\begin{equation*}
\frac{M_{n,[0,h]}}{n} \to 1+G,
\end{equation*}
almost surely as $n \to +\infty$, so that by (\ref{eq:sum1g3}) it
suffices to show that
\begin{equation*}
\liminf_{n \to +\infty} \frac{M_{n,[a,b]}}{n} \geq \nu_{[a,b]},
\end{equation*}
almost surely for all $0 \leq a < b \leq h$.

Let $1\leq I < +\infty$ and consider once again the discretization
of the $(\fcal,\qcal)$-chain. Let $\tildenu^I_j$ be as in
Lemma~\ref{lem:couplingalonecont} (we now indicate the dependence
on $I$ because we will need to take $I \to +\infty$). Similarly,
let $\tildelambda^I_0$ be as in (\ref{eq:bediscr}). Fix $0 \leq a
< b \leq h$. Let $\kcal^I$ be the largest subset of
$\{1,\ldots,I\}$ such that
\begin{equation*}
\bigcup_{i \in \kcal^I} (\undf_i^I, \ovf_i^I) \subseteq [a,b].
\end{equation*}
By the coupling lemma, we have
\begin{eqnarray*}
\liminf_{n \to +\infty} \frac{M_{n,[a,b]}}{n}
&\geq& \sum_{i\in \kcal^I} \tildenu_i^I \\
&=& \sum_{i \in \kcal^I} \frac{\tildeq_i^I}{\tildelambda_0^I - \undf_i^I}\\
&\geq& \sum_{i \in \kcal^I} \int_{\undf_i^I}^{\ovf_i^I} \frac{g(x)}{\tildelambda_0^I + \eps - x} \diff x\\
&\geq& \int_{a + \eps}^{b - \eps} \frac{g(x)}{\tildelambda_0^I +
\eps - x} \diff x.
\end{eqnarray*}
Since $\eps = 1/I$ goes to $0$ as $I \to +\infty$, it suffices to
prove
\begin{equation}\label{eq:lambda0cont}
\tildelambda^I_0 \to \lambda_0,
\end{equation}
as $I \to +\infty$.

We first show that $\tildelambda_0 > \lambda_0 - \eps$. Indeed,
assume $\tildelambda_0^I = \lambda_0 - \eps$. Then, the sum in
(\ref{eq:bediscr}) satisfies
\begin{eqnarray*}
\sum_{j=1}^{I} \frac{\ovf_j^I \tildeq_j^I}{\tildelambda_0^I -
\undf_j^I} + h \frac{(1+G)(\tildelambda_0^I)^{-1}\eps}
{\tildelambda_0^I - (h - \eps)}
&>& \sum_{j=1}^{I} \frac{\ovf_j^I \tildeq_j^I}{\tildelambda_0^I - \undf_j^I} \\
&\geq& \int_{0}^{h} \frac{x g(x)}{\lambda_0 - x}\\
&=& 1,
\end{eqnarray*}
which proves the claim, by monotonicity.

Take any $\ovlambda_0 > \lambda_0$. We show that eventually,
$\tildelambda_0^I < \ovlambda_0$. Let
\begin{equation*}
\ical(\lambda) = \int_{0}^h \frac{x g(x)}{\lambda - x}\diff x,
\end{equation*}
and note that $\ical(\ovlambda_0) < 1$. From (\ref{eq:sum1g3}), we
get
\begin{eqnarray*}
\sum_{j=1}^{I} \frac{\ovf_j^I \tildeq_j^I}{\ovlambda_0 -
\undf_j^I} &=& \eps\sum_{j=1}^{I} \frac{\tildeq_j^I}{\ovlambda_0 -
\undf_j^I}
+ \sum_{j=1}^{I} \frac{\undf_j^I \tildeq_j^I}{\ovlambda_0 - \undf_j^I}\\
&\leq& \eps \int_{0}^h \frac{g(x)}{\ovlambda_0 - x} + \int_{0}^h \frac{x g(x)}{\ovlambda_0 - x}\\
&\leq& \eps(1+G)(\lambda_0)^{-1} + \ical(\ovlambda_0).
\end{eqnarray*}
As for the other term in (\ref{eq:sum1g}), note that as soon as
\begin{equation*}
\ovlambda_0 \geq \lambda_0(1 + (1+G)(\lambda_0)^{-1}\sqrt{\eps})
\geq h(1 + (1+G)(\ovlambda_0)^{-1}\sqrt{\eps}) - \eps,
\end{equation*}
(the second inequality is always true), we have
\begin{eqnarray*}
h \frac{(1+G)(\ovlambda_0)^{-1}\eps} {\ovlambda_0 - (h - \eps)}
\leq \sqrt{\eps}.
\end{eqnarray*}
Therefore,
\begin{eqnarray*}
\sum_{j=1}^{I} \frac{\ovf_j^I \tildeq_j^I}{\ovlambda_0 -
\undf_j^I} + h \frac{(1+G)(\ovlambda_0)^{-1}\eps} {\ovlambda_0 -
(h - \eps)} \leq \sqrt{\eps} + \eps(1+G)(\lambda_0)^{-1} +
\ical(\ovlambda_0) < 1,
\end{eqnarray*}
for $I$ large enough, which proves the claim by (\ref{eq:sum1g})
and monotonicity. Furthermore, since $\ovlambda_0 > \lambda_0$ is
arbitrary, we have (\ref{eq:lambda0cont}). This concludes the
proof.
\end{proof}

\subsection{Innovation-Pays-Off Phase}\label{sec:becont}

Assume the density $g$ is defined on $[0,h]$ with $h < +\infty$
and assume further that $g(x) > 0$ for all $x \in (0,h)$ (we allow
$0$ at the endpoints). In this section, we consider the case
\begin{equation}\label{eq:be3cont}
\ical(h) := \int_{0}^h \frac{x g(x)}{h - x}\diff x < 1.
\end{equation}
We also assume
\begin{equation}\label{eq:intone}
\int_0^h g(x) \diff x = 1,
\end{equation}
although this is not necessary.
\begin{example}\label{example:beta}
Consider the case where $\qcal$ is Beta($\alpha$,$\beta$). Then it
is easy to show that
\begin{eqnarray*}
\int_{0}^1 \frac{x g(x)}{1 - x}\diff x &=& \frac{B(\alpha + 1,
\beta -1)}{B(\alpha, \beta)} = \frac{\alpha}{\beta - 1},
\end{eqnarray*}
where $B$ is the Beta function. Therefore, (\ref{eq:be3cont}) is
satisfied if $\beta > \alpha + 1$. This example is a
generalization of the example given in~\cite{BianconiBarabasi:01}.
\end{example}

By (\ref{eq:be3cont}), there is no solution $\lambda_0 \geq h$ to
\begin{equation*}
\int_{0}^h \frac{x g(x)}{\lambda_0 - x}\diff x = 1.
\end{equation*}
Instead, for $0 \leq a \leq b \leq h$, let
\begin{equation*}
\nu_{[a,b]} = h \int_{a}^b \frac{g(x)}{h - x} \diff x.
\end{equation*}
Note in particular that
\begin{equation*}
\nu_{[0,h]} = \int_{0}^h (h - x) \frac{g(x)}{h - x} \diff x +
\int_{0}^h \frac{x g(x)}{h - x} \diff x = 1 + \ical(h) < 2.
\end{equation*}
Also, for $n \geq 0$, let $M_{n,[a,b]}$ be the number of edges
with an endpoint of fitness in $[a,b]$ (counting twice edges with
two endpoints of fitness in $[a,b]$). For ease of notation, we
note $M_{n,x} := M_{n,[x,x]}$.

We prove the following.
\begin{theorem}[Continuous Case: Innovation-Pays-Off Phase]\label{thm:boseeinsteincont}
Assume $g$ is defined on $[0,h]$ with $h < +\infty$ and assume
further that $g(x) > 0$ for all $x \in (0,h)$ and
\begin{equation*}
\int_{0}^h \frac{x g(x)}{h - x}\diff x < 1.
\end{equation*}
Then it holds that for all $0 \leq a < b < h$,
\begin{equation}\label{eq:becontres1}
\frac{M_{n,[a,b]}}{n} \to \nu_{[a,b]},
\end{equation}
almost surely as $n \to +\infty$. Moreover, for $0 \leq a \leq h$,
we have
\begin{equation}\label{eq:becontres2}
\frac{M_{n,[a,h]}}{n} \to 2 - \nu_{[0,a]},
\end{equation}
almost surely as $n \to +\infty$.
\end{theorem}
\begin{proof}
The convergence (\ref{eq:becontres2}) follows trivially from
(\ref{eq:becontres1}). Also, from the proof of
Theorem~\ref{thm:fitgetrichcont} it follows that
\begin{equation*}
\liminf_{n \to +\infty} \frac{M_{n,[a,b]}}{n} \geq \nu_{[a,b]},
\end{equation*}
almost surely for all $0 \leq a < b \leq h$ (replace $\lambda_0$
with $h$ in the proof).

To obtain an upper bound, we consider the modified chain with
fitness distribution $\qcal_\eps$ with
\begin{equation*}
g_\eps(x) = \left\{
\begin{array}{ll}
g(x), & 0 \leq x \leq h - \eps,\\
0, & x > h-\eps.
\end{array}
\right.
\end{equation*}
It is clear that we can couple this modified chain with the
original one so that for all $0 \leq a \leq b \leq h-\eps$
\begin{equation*}
M_{n,[a,b]} \leq M^{(\eps)}_{n,[a,b]}.
\end{equation*}
(Proceed similarly to the proof of Lemma~\ref{lem:couplingalone}.)
Also, from Example~\ref{ex:fitgetrichcont}, it follows that the
modified chain is in the Fit-Get-Richer phase which allows to
apply Theorem~\ref{thm:fitgetrichcont} (this is the reason we
allowed $G < 1$ in the proof of Theorem~\ref{thm:fitgetrichcont}).
Therefore, for all $0 \leq a \leq b \leq h-\eps$,
\begin{equation*}
\limsup_{n \to +\infty} \frac{M_{n,[a,b]}}{n} \leq
\lambda_0^{(\eps)} \int_{a}^b \frac{g(x)}{\lambda^{(\eps)}_0 -
x}\diff x,
\end{equation*}
where $\lambda_0^{(\eps)} \geq h-\eps$ is a solution to
\begin{equation*}
\int_0^{h-\eps} \frac{xg(x)}{\lambda_0^{(\eps)} - x}\diff x = 1.
\end{equation*}

We claim that $\lambda_0^{(\eps)} \to h$ as $\eps \to 0$ which
proves (\ref{eq:becontres1}). Indeed, note that
\begin{eqnarray*}
\int_{0}^{h-\eps} \frac{xg(x)}{h - x} \leq \int_{0}^{h}
\frac{xg(x)}{h - x} < 1.
\end{eqnarray*}
Therefore, $h - \eps \leq \lambda_0^{(\eps)} < h$. This concludes
the proof.
\end{proof}

\subsection{Unbounded Case}\label{sec:app:continuous_unbounded}

The unbounded fitness case also follows easily from the previous
proof (see also the proof in the discrete case). Therefore, we
state the result without proof.
\begin{theorem}[Continuous Case: Unbounded Case]\label{thm:unboundedcont}
Assume $g$ is defined on $[0,+\infty)$. Assume further that $g(x)
> 0$ for all $x \in (0,+\infty)$ and
\begin{equation*}
\int_0^{+\infty} g(x) = 1.
\end{equation*}
Then it holds that for all $0 \leq a < b < +\infty$,
\begin{equation*}
\frac{M_{n,[a,b]}}{n} \to \int_a^b g(x) \diff x,
\end{equation*}
almost surely as $n \to +\infty$. Moreover, for $0 \leq a <
+\infty$, we have
\begin{equation*}
\frac{M_{n,[a,+\infty)}}{n} \to 2 - \int_0^a g(x) \diff x,
\end{equation*}
almost surely as $n \to +\infty$.
\end{theorem}

\end{appendix}



\end{document}